\documentclass[11pt,reqno]{amsart}

\usepackage[a4paper,margin=1in]{geometry}
\usepackage{amsmath,amssymb,amsthm,mathtools}
\usepackage{enumitem}
\usepackage{tikz-cd}
\usepackage{mathrsfs}
\usepackage{extarrows}
\usepackage[colorlinks=true,linkcolor=blue,citecolor=blue,urlcolor=blue]{hyperref}

\setlist{nosep}
\newtheorem{theorem}{Theorem}[section]
\newtheorem{proposition}[theorem]{Proposition}
\newtheorem{lemma}[theorem]{Lemma}
\newtheorem{corollary}[theorem]{Corollary}

\newtheorem{conjecture}[theorem]{Conjecture}

\theoremstyle{remark}
\newtheorem{remark}[theorem]{Remark}
\newtheorem{example}[theorem]{Example}

\DeclareMathOperator{\Spec}{Spec}
\DeclareMathOperator{\Tot}{Tot}
\DeclareMathOperator{\Sym}{Sym}
\DeclareMathOperator{\Chow}{Chow}
\DeclareMathOperator{\ord}{ord}
\DeclareMathOperator{\Pic}{Pic}
\DeclareMathOperator{\Cl}{Cl}
\DeclareMathOperator{\cm}{cm}
\DeclareMathOperator{\pr}{pr}

\newcommand{\A}{\mathbb A}
\newcommand{\C}{\mathbb C}
\newcommand{\OO}{\mathcal O}
\newcommand{\sB}{\mathfrak B}
\newcommand{\sd}{\operatorname{sd}}

\numberwithin{equation}{theorem}

\setlist[enumerate]{label=(\thetheorem.\arabic*), before={\setcounter{enumi}{\value{equation}}}, after={\setcounter{equation}{\value{enumi}}}}

\title[]
{Local obstructions to the surjectivity of spectral morphisms}

\author{Siqi He}
\address{Siqi He, Institute of Mathematics, Academy of Mathematics and Systems Science, Chinese Academy of Sciences, Beijing 100190, China}
\email{\href{mailto:sqhe@amss.ac.cn}{sqhe@amss.ac.cn}}

\author{Jie Liu}
\address{Jie Liu, Institute of Mathematics, Academy of Mathematics and Systems Science, Chinese Academy of Sciences, Beijing 100190, China}
\email{\href{mailto:jliu@amss.ac.cn}{jliu@amss.ac.cn}}

\author{Siqing Zhang}
\address{Siqing Zhang, Max Planck Institute for Mathematics, Vivatsgasse 7, 53111 Bonn, Germany}
\email{\href{mailto:siqing.zhang.math@gmail.com}{siqing.zhang.math@gmail.com}}

\begin{document}

\begin{abstract}
    Chen and Ng\^o showed that the Hitchin morphism for Higgs bundles on a projective manifold $X$ factors through a closed subscheme of the Hitchin base called the spectral base, and the induced map is called the spectral morphism. In this paper, we investigate local Cohen--Macaulay obstructions to the surjectivity of spectral morphisms via a refined spectral correspondence. We show that every affine normal variety with a local ring that admits no rank one maximal Cohen--Macaulay modules can be geometrically realized as a local obstruction to the surjectivity of spectral morphisms, and that these local obstructions can be compactified via the Kawamata covering trick.
	
    As an application of these constructions, we provide counterexamples to the Chen--Ng\^o conjecture on the surjectivity of spectral morphisms for $\operatorname{GL}_r$-Higgs bundles in dimensions $n\geq 9$. Moreover, we construct counterexamples to the Chen--Ng\^o conjecture for $\operatorname{SL}_2$-Higgs bundles in every dimension $n\geq 2$ and for semistable $\operatorname{GL}_3$-Higgs bundles on a smooth projective surface.
\end{abstract}

\maketitle
\tableofcontents

\section{Introduction}

Throughout this paper, we work over the field $\mathbb C$ of complex numbers.

\subsection{Motivation}

Moduli spaces of Higgs bundles and the Hitchin morphism occupy a central place in modern geometry. The nonabelian Hodge correspondence relates representations of the fundamental group of a projective manifold to polystable Higgs bundles with vanishing Chern classes \cite{Corlette1988,Donaldson1987,Hitchin1987SelfDuality,Simpson1988}. These moduli spaces have become important tools in the study of character varieties \cite{Simpson1991,Simpson1992,Simpson1994a,Simpson1997}. For curves, the Hitchin morphism sends a Higgs bundle of rank $r$ to the characteristic coefficients of its Higgs field \cite{Hitchin1987}; the corresponding spectral curve describes the Hitchin fibre via the Beauville--Narasimhan--Ramanan correspondence \cite{BeauvilleNarasimhanRamanan1989}.

In higher dimensions, the situation is more complicated. Let $X$ be a projective manifold. A \emph{$\operatorname{GL}_r$-Higgs bundle} on $X$ is a pair $(\mathscr E,\theta)$, where $\mathscr E$ is a vector bundle of rank $r$ and
\[
\theta\colon \mathscr E\longrightarrow
\mathscr E\otimes\Omega_X^1
\]
is an $\mathscr O_X$-linear morphism satisfying
$\theta\wedge\theta=0$. Locally, writing
$\theta=\sum_i\theta_i\,dz_i$, this condition is equivalent to
$[\theta_i,\theta_j]=0$ for all $i,j$. Thus the characteristic
coefficients of $\theta$ satisfy additional algebraic relations and need not fill the \emph{Hitchin base}
\[
A_X^r\coloneqq
\bigoplus_{i=1}^r H^0\bigl(X,\operatorname{Sym}^i\Omega_X^1\bigr).
\]
Chen and Ng\^o introduced a closed subscheme $S_X^r\subset A_X^r$, called the \emph{spectral base}, through which the
Hitchin morphism factors (see \S\,\ref{ss.Spectralbase} for
$\operatorname{GL}_r$-Higgs bundles and \cite[\S\,5]{ChenNgo2020} for reductive groups and $G$-Higgs bundles). Let $\mathcal H_X^r$ denote the set of isomorphism classes of Higgs bundles of rank $r$ on $X$. The resulting spectral morphism is
\[
\sd_X^r\colon\mathcal H_X^r\longrightarrow S_X^r,
\]
which sends a Higgs bundle to the characteristic coefficients of its Higgs field.

The central conjecture of Chen and Ng\^o is the following. We state it in
the $\operatorname{GL}_r$ case; see \cite{ChenNgo2020} for the formulation
for a general reductive group.

\begin{conjecture}[\protect{\cite[Conjecture 5.2]{ChenNgo2020}}]
\label{conj:chen-ngo-global}
Let $X$ be a projective manifold of dimension $n\geq 2$. For every $r\geq 2$, the spectral morphism $\sd_X^r\colon\mathcal H_X^r\longrightarrow S_X^r$ is surjective.
\end{conjecture}

For $s\in S_X^r$, write $\mathcal H_X^r(s)$ for the fibre of
$\sd_X^r$ over $s$. There has been substantial progress on
Conjecture~\ref{conj:chen-ngo-global} in recent years.

\begin{itemize}
\item Chen and Ng\^o proved that $\mathcal H_X^r(s)\neq\emptyset$ for generically regular spectral data $s$ and $\operatorname{GL}_r$-Higgs bundles on smooth projective surfaces \cite{ChenNgo2020}.

\item Song and Sun proved the conjecture for $\operatorname{GL}_r$-Higgs bundles on every smooth projective surface \cite{SongSun2024}.

\item The first two authors proved the conjecture for $\operatorname{GL}_2$-Higgs bundles on projective manifolds of arbitrary dimension \cite{HeLiu2024}.

\item Huynh proved the conjecture for general reductive groups on ruled surfaces and certain elliptic fibred surfaces \cite{Huynh2025}.

\item Patel and Weissmann obtained positive results for hyperelliptic and $K$-trivial varieties, and more generally for $r$-small varieties \cite{PatelWeissmann2026b,PatelWeissmann2026a}.

\item Sheshmani, Wang, and Xia proved the conjecture for
$\operatorname{SL}_{2m+1}$-Higgs bundles on algebraic surfaces and treated
the general $\operatorname{SL}_r$ and $\operatorname{Sp}_{2r}$ cases for
products of two curves
\cite{SheshmaniWangXia2026}.

\item Sun proved the non-emptiness of the fibres of $\operatorname{sd}_X^r$ over cyclic spectral data \cite{Sun2026}.
\end{itemize}

There is also related work on the spectral base from the viewpoints of
locally symmetric varieties and commuting schemes
\cite{HeLiuMok2024,SheshmaniXiaYuan2025,SongXiaXu2023}.

\subsection{Main results}

Our results concern the Chen--Ng\^o conjecture and its semistable analogue.
We first construct counterexamples to the conjecture for
$\operatorname{GL}_r$-Higgs bundles on projective manifolds.

\begin{theorem}\label{thm:Chen-Ngo-false}
Conjecture~\ref{conj:chen-ngo-global} is false. More precisely, there exists
an integer $r\geq 3$, a projective manifold $X$, and a generically regular
spectral datum $s\in S_X^r$ such that
$\mathcal H_X^r(s)=\emptyset$.
\end{theorem}

More precisely, we prove that for every $d\geq 9$, there exists a projective manifold $X$ of dimension $d$, an integer $3\leq r\leq 44$, and a generically regular spectral datum $s\in S_X^{r}$ such that
$\mathcal H_X^{r}(s)=\emptyset$; see Remark~\ref{rmk.Min-Possible}. In view of the surface case \cite{ChenNgo2020,SongSun2024} and the rank two case
\cite{HeLiu2024}, these examples may not be optimal, and Conjecture \ref{conj:chen-ngo-global} may still hold in low rank or low dimension.

We next consider $\operatorname{SL}_2$-Higgs bundles. These are
$\operatorname{GL}_2$-Higgs bundles with trivial determinant and traceless Higgs field. Although the Chen--Ng\^o conjecture holds for $\operatorname{GL}_2$-Higgs bundles \cite{HeLiu2024}, the following result shows that it fails for
$\operatorname{SL}_2$-Higgs bundles in every dimension at least two; see \cite{SheshmaniWangXia2026} for the case of $\operatorname{SL}_r$-Higgs bundles with $r$
odd.

\begin{theorem}
\label{thm.SL_2}
For every $d\geq 2$, there exists a projective manifold $X$ of dimension $d$ and a generically regular spectral datum
$s\in S_X^{\operatorname{SL}_2}$ such that $\mathcal H_X^{\operatorname{SL}_2}(s)=\emptyset$.
\end{theorem}

We refer the reader to \S\,\ref{ss.Sl2-Setup} for the explicit definitions of $S_X^{\operatorname{SL}_2}$ and $\mathcal{H}_X^{\operatorname{SL}_2}(s)$. We construct three families of counterexamples for
$\operatorname{SL}_2$-Higgs bundles: two endoscopic families, on canonically
polarized projective manifolds of dimension at least $2$
(Theorem~\ref{thm:sl2-counterexample}) and on $K$-trivial projective
manifolds of dimension at least $3$
(Theorem~\ref{thm: k trivial counterexamples}), and a non-endoscopic family
in dimensions $n\geq 3$ (Theorem~\ref{thm: non-endo}).

Finally, we consider the analogous question for semistable $\operatorname{GL}_r$-Higgs bundles. Given a polarized projective manifold $(X,H)$, let
$\mathcal H_{X,H}^{r,\mathrm{ss}}(s)$ denote the set of isomorphism classes
of $\mu_H$-semistable Higgs bundles of rank $r$ and with spectral datum $s$. In \cite{HeLiu2024}, it was proved that
$\mathcal H_{X,H}^{2,\mathrm{ss}}(s)\neq\emptyset$ for every
$s\in S_X^2$. The following result shows that this fibre can nevertheless be empty for surfaces and rank $r=3$.

\begin{theorem}\label{thm:polystable-counterexample}
There exists a polarized smooth projective surface $(X,H)$ and a generically regular spectral datum $s\in S_X^3$ such that $\mathcal H_{X,H}^{3,\mathrm{ss}}(s)=\emptyset$.
\end{theorem}

\subsection{Strategy and organization}

We now describe the strategy of the proofs. Our starting point is the spectral correspondence \cite[Proposition~6.3]{ChenNgo2020} (cf.~\cite{BeauvilleNarasimhanRamanan1989})
and its refinement in Proposition~\ref{prop:general-higgs-factors-through-Z}; see also \cite[Theorem~1.2]{HeLiu2024}. For a generically regular spectral
datum $s\in S_X^r$, the refined correspondence describes the fibre $\mathcal H_X^r(s)$ in terms of rank one maximal Cohen--Macaulay sheaves on suitable finite small modifications of the spectral variety $X_s$. Thus the absence of rank one maximal Cohen--Macaulay modules over local rings of finite small modifications of $X_s$ provides local obstructions to the non-emptiness of $\mathcal H_X^r(s)$.

Ma constructed in \cite{Ma2019} normal affine varieties $Y$ with a unique singular point $o_Y$ such that neither the local ring $\mathscr O_{Y,o_Y}$ nor its completion $\widehat{\mathscr O}_{Y,o_Y}$ admits a rank one maximal
Cohen--Macaulay module. Using a Noether normalization and replacing $Y$ by $W=Y\times\mathbb A^m$ if necessary, we realize the pointed germ of $W$ by an exact Lagrangian embedding in codimension one into the cotangent bundle $T^*\mathbb{A}^{N-1}$ of an affine space $\mathbb{A}^{N-1}$, where $N=\dim Y+m+1$. Its image determines a non-zero generically regular datum
$s\in S_{\mathbb{A}^{N-1}}^r$, where $r$ is the degree of the projection $W\to \mathbb{A}^{N-1}$. The refined spectral correspondence then gives
$\mathcal H_{\mathbb{A}^{N-1}}^r(s)=\emptyset$.

To globalize these local models, it suffices to construct a finite small modification of a spectral variety, which contains one of them as an analytic germ. After adjoining an affine elliptic curve factor $E$, we obtain a finite cover
$q_M\colon M\coloneqq W\times E\to\mathbb A^N$ of degree $2r$ and a rational section $\sigma_M$ of $q_M^*\Omega_{\mathbb A^N}^1$ such that
$(dq_M)(\sigma_M)$ is a regular $1$-form on $M$ extending to the smooth loci of normal projective birational models of $M$. Moreover, the local obstructions on $W$ also yield germs of $M$ as local obstructions to the existence of rank one maximal Cohen--Macaulay modules. We then compactify $\mathbb A^N$ to $\mathbb P^N$ naturally and the Kawamata's covering trick is applied to get a suitable smooth finite cover $X$ of a log resolution $\widetilde{\mathbb{P}}$ of $\mathbb{P}^N$, which is \'etale near a chosen point. The finite cover $q_M\colon M\rightarrow \mathbb{A}^N$ yields a natural finite cover $\gamma\colon S\rightarrow  X$ such that the induced pullback of $\sigma_M$ on $S$ is a regular section of $\gamma^*\Omega_X^1$ and defines a generically regular spectral datum $s\in S_X^{r'}$ for $r'=\deg\gamma\leq 2r$. In particular, since the cover $X\rightarrow \widetilde{\mathbb{P}}$ is \'etale near the chosen point, the variety $S$ contains the prescribed analytic germ in $M$. The spectral correspondence, together with the local obstructions on $S$, gives the desired projective counterexamples.

The arguments for the $\operatorname{SL}_2$ and polystable cases are analogous. We construct finite covers of projective manifolds, which can be realized as a spectral cover (or its small finite modifications). Then we show that the corresponding extremal conditions cannot be satisfied under the refined spectral correspondence.

The paper is organized as follows. Section~\ref{sec.Prelimi} collects basic definitions and facts about spectral bases, spectral morphisms, and the refined spectral correspondence. Section~\ref{ss.Local-Models} constructs the local models, following \cite{Ma2019}. Section~\ref{sec.global-counterexamples} gives the compactification construction and proves Theorem~\ref{thm:Chen-Ngo-false}. Section~\ref{sec:sl2-counterexample} treats the $\operatorname{SL}_2$ case and proves Theorem~\ref{thm.SL_2}. Section~\ref{sec:polystable} treats the rank $3$ semistable case and proves
Theorem~\ref{thm:polystable-counterexample}. Finally, Appendix~\ref{sec: dim 2} proves an equivariant version of the classical Noether--Lefschetz theorem used in Section~\ref{sec:sl2-counterexample}.

\subsection*{Conventions}
Let $X$ be a variety and let $\mathscr{E}$ be
a vector bundle on $X$.  The projectivization $\mathbb P(\mathscr{E})$ of $\mathscr{E}$ parametrizes lines in the fibers
of $\mathscr{E}\rightarrow X$. On the other hand, the total space of $\mathscr{E}$ is defined as:
\[
\Tot(\mathscr{E})
\coloneqq
\Spec_X\!\left(\bigoplus_{i\geq0}\Sym^i\mathscr{E}^\vee\right),
\]
where $\mathscr{E}^{\vee}$ is the dual bundle of $\mathscr{E}$. In particular, if $X$ is smooth, the total space of the cotangent bundle $\Omega_X^1$ of $X$ is denoted by $T^*X$ with $\pi\colon T^*X\rightarrow X$ the natural projection.

\subsection*{Acknowledgments}
S.~Zhang is grateful to Matthew Huynh for useful discussions.
	  J.~Liu is supported by the CAS Project for Young Scientists in Basic Research (No. YSBR-033), the Youth Innovation Promotion Association CAS and the NSFC grant (No. 12288201). S.~He is partially supported by NSFC grant No.12288201 and No.2023YFA1010500.
      S.~Zhang is supported by an AMS-Simons travel grant.

\subsection*{Declaration of AI}
The authors used GPT-5.6 Sol in developing ideas and exploring
proof strategies for this work. The manuscript was written
entirely by the authors, who verified the mathematical arguments
and take full responsibility for its content.

\section{Preliminaries}\label{sec:hitchin-local}
\label{sec.Prelimi}

In this section, we collect basic definitions and facts concerning Higgs bundles, especially the spectral base introduced by Chen and Ng\^o in \cite{ChenNgo2020}.

\subsection{Higgs bundles and the Hitchin morphism}

Let $X$ be a smooth variety of dimension $n$.  A rank-$r$ \emph{Higgs bundle} on $X$ is a pair
$(\mathscr{E},\theta)$, where $\mathscr{E}$ is a vector bundle of rank $r$ over $X$ and
\[
  \theta\colon\mathscr{E}\longrightarrow
  \mathscr{E}\otimes\Omega_X^1
\]
is an $\mathscr{O}_X$-linear morphism satisfying the integrability condition $\theta\wedge\theta=0$.  Equivalently,
if $x_1,\ldots,x_n$ are local coordinates on $X$ and
$\theta=\sum_{i=1}^n\theta_i\,dx_i$, then
$[\theta_i,\theta_j]=0$ for all $i,j$.  Thus a Higgs field is locally a tuple
of pairwise commuting endomorphisms of $\mathscr{E}$. Let $\mathcal H_X^r$ denote the set of isomorphism classes of rank-$r$ Higgs
bundles on $X$.  The rank-$r$ \emph{Hitchin base} of $X$ is
\[
  A_X^r\coloneqq
  \bigoplus_{i=1}^r H^0\!\left(X,\Sym^i\Omega_X^1\right).
\]
The characteristic polynomial of a Higgs field $\theta$ has the form
\[
  \det(T-\theta)
  =T^r-s_1(\theta)T^{r-1}+s_2(\theta)T^{r-2}
   -\cdots+(-1)^rs_r(\theta),
\]
where
$s_i(\theta)\in H^0(X,\Sym^i\Omega_X^1)$.  This defines the \emph{Hitchin morphism}
\[
  h_X^r\colon\mathcal H_X^r\longrightarrow A_X^r,
  \qquad
  (\mathscr{E},\theta)\longmapsto
  \bigl(s_1(\theta),\ldots,s_r(\theta)\bigr).
\]
When $X$ is a smooth projective curve, the integrability condition is automatic and it is known that the Hitchin morphism $h_X^r$ is surjective. However, in higher
dimensions the local endomorphisms must commute, and this imposes additional
algebraic relations on the characteristic coefficients.  Thus one could not
expect $h_X^r$ to be surjective onto $A_X^r$.

\subsection{The spectral base and the spectral morphism}
\label{ss.Spectralbase}

Chen and Ng\^o define a closed subscheme $S_X^r\subset A_X^r$, called the
\emph{spectral base}. In the situation of usual Higgs bundles (i.e. $\operatorname{GL}_r$-Higgs bundles), it has a concrete description in terms of relative Chow varieties.  Let $V$ be a vector space. Recall that the map
\[
  \begin{aligned}
  \Chow^r(V)&\longrightarrow
  V\times\Sym^2V\times\cdots\times\Sym^rV,\\
  [v_1,\ldots,v_r]&\longmapsto
  \bigl(\sigma_1(v_1,\ldots,v_r),\ldots,
        \sigma_r(v_1,\ldots,v_r)\bigr)
  \end{aligned}
\]
is a closed embedding, where $\sigma_i$ is the $i$th elementary symmetric
function.  Applied fiberwise to $T^*X\rightarrow X$, this gives a closed embedding
\[
  \Chow^r\bigl(T^*X/X\bigr)
  \hookrightarrow
  T^*X\times_X\dots\times_X \Tot \bigl(\Sym^r\Omega_X^1\bigr).
\]
The spectral base $S_X^r$ then is the space of sections of
\[
  \Chow^r\bigl(T^*X/X\bigr)\longrightarrow X.
\]
Equivalently, $s=(s_1,\ldots,s_r)\in A_X^r$ belongs to $S_X^r$ if and only if,
for every $x\in X$, there exist covectors
$\alpha_1,\ldots,\alpha_r\in\Omega_{X,x}^1$ (not necessarily pairwise distinct) such that
\[
  s_i(x)=\sigma_i(\alpha_1,\ldots,\alpha_r)
  \qquad(1\leq i\leq r).
\]
An element of $S_X^r$ is called a \emph{spectral datum} and it is called
\emph{generically regular} if its zero-cycle over the generic point of $X$
consists of $r$ distinct covectors.

The integrability condition $\theta\wedge \theta=0$ implies that $h_X^r$ factors through $S_X^r$ \cite[Proposition 2.8]{HeLiuMok2024}: Indeed, at $x\in X$ write
$\theta(x)=\sum_{i=1}^n \theta_i(x)\,dx_i$.  Since the matrices $\theta_i(x)$ commute, they can
be simultaneously upper triangularized.  Their diagonal entries give an
unordered $r$-tuple of covectors, whose elementary symmetric functions are
precisely the characteristic coefficients of $\theta$.  We therefore have a
factorization
\[
  \begin{tikzcd}[column sep=large]
  \mathcal H_X^r \arrow[r,"{\sd_X^r}"] \arrow[dr,"{h_X^r}"']
    & S_X^r \arrow[d,hook,"\iota_X^r"]\\
    & A_X^r .
  \end{tikzcd}
\]
The map $\sd_X^r$ is called the \emph{spectral morphism}.  For
$s\in S_X^r$, we write $\mathcal H_X^r(s)$ for its fiber over $s$.

\begin{lemma}[{\cite[Lemma 2.6]{HeLiuMok2024}}]
\label{lem:dense-open-spectral-base}
Let $X$ be a smooth variety and let $s\in A_X^r$.  Suppose that there is a
dense open subset $X^\circ\subset X$ such that $s|_{X^{\circ}}\in S_{X^{\circ}}^r$.  Then $s\in S_X^r$.  

Moreover, if we assume furthermore that $s|_{X^{\circ}}$ is generically regular, then $s$ itself is generically regular.
\end{lemma}

\subsection{Spectral varieties}\label{sec:spectral-varieties}

Let $s=(s_1,\ldots,s_r)\in S_X^r$, and let
$\lambda_{\mathrm{taut}}\in H^0(T^*X,\pi^*\Omega_X^1)$ be the tautological
$1$-form.  The \emph{spectral variety} associated with $s$ is the closed
subscheme $X_s\subset T^*X$ cut out by
\[
  \lambda_{\mathrm{taut}}^r
  -\pi^*s_1\lambda_{\mathrm{taut}}^{r-1}
  +\pi^*s_2\lambda_{\mathrm{taut}}^{r-2}
  -\cdots+(-1)^r\pi^*s_r=0,
\]
viewed as an element of $H^0(T^*X,\pi^*\Sym^r\Omega_X^1)$.

For an affine description, suppose that $X=\Spec A$ and that the cotangent
bundle is free, with basis $dx_1,\ldots,dx_n$.  Then
\[
  T^*X=\Spec A[\xi_1,\ldots,\xi_n],
  \qquad
  \lambda_{\mathrm{taut}}=\sum_{i=1}^n\xi_i\,dx_i.
\]
Write
\[
  \lambda_{\mathrm{taut}}^r-s_1\lambda_{\mathrm{taut}}^{r-1}
  +\cdots+(-1)^rs_r
  =\sum_{|\alpha|=r}F_\alpha\,dx^\alpha.
\]
Then we have
\[
  X_s=\Spec \biggl(A[\xi_1,\ldots,\xi_n]\bigg/
                    \langle F_\alpha\mid |\alpha|=r\rangle\biggr).
\]
On the open set where $s$ is represented by $r$ distinct $1$-forms, the spectral variety $X_s$ is the disjoint union of their graphs.  In particular,
if $s$ is generically regular, then $X_s\to X$ is finite of degree
$r$.  However, the scheme structure of $X_s$ may be complicated over the ramification and non-flat loci of the finite cover $X_s\rightarrow X$.

\subsection{MCM sheaves and refined spectral correspondence}

For a variety $Y$, we denote by $\mathcal{M}_{\cm}(Y)$ the set of isomorphism classes of rank one maximal Cohen--Macaulay (MCM for short) sheaves on $Y$. We refer the reader to \cite[IV, B)]{Serre1965} for the basic facts on Cohen--Macaulay modules. 

Suppose in addition that $Y$ is normal, then we can identify the divisor class group $\Cl(Y)$ of $Y$ with the set of isomorphism classes of divisorial sheaves on $Y$; that is, $\mathscr{O}_Y(D)$ for a Weil divisor $D$ on $Y$. Then, in particular, there exists a natural injection $\mathcal{M}_{\cm}(Y)\rightarrow \Cl(Y)$. Indeed, for $\mathscr{L}\in \mathcal{M}_{\cm}(Y)$, the restriction $\mathscr{L}|_{Y_{\operatorname{reg}}}$ is invertible by \cite[IV, D), Corollaire 2]{Serre1965}, thus there exists a divisor $D$, unique up to linear equivalence, on $Y$ such that $\mathscr{O}_{Y_{\operatorname{reg}}}(D)\cong \mathscr{L}|_{Y_{\operatorname{reg}}}$, which then implies $\mathscr{O}_Y(D)\cong \mathscr{L}$ because both $\mathscr{O}_Y(D)$ and $\mathscr{L}$ are reflexive and $\operatorname{codim}(Y\setminus Y_{\operatorname{reg}})\geq 2$.

Let $X$ be a projective manifold and let $s\in S_X^r$ be a generically regular spectral datum. The spectral correspondence canonically identifies $\mathcal H_X^r(s)$ with $\mathcal{M}_{\cm}(X_s)$ \cite[Proposition 6.3]{ChenNgo2020}. We need the following variant.

\begin{proposition}\label{prop:general-higgs-factors-through-Z}
Let $X$ be a smooth variety of dimension $n$, and let
$p\colon Y\to X$ be a finite surjective morphism of degree $r$, with $Y$ a (possibly reducible) equi-dimensional normal variety.  Suppose
that there is a morphism $f\colon Y\to T^*X$ such that $\pi\circ f=p$ and such
that $f$ is an embedding on an open subset of $Y$ whose complement in $Y$ has codimension at least two.  Then $f$ determines a generically regular spectral datum $s\in S_X^r$ and there exists a natural bijection $\mathcal{M}_{\cm}(Y)\rightarrow \mathcal{H}_X^r(s)$.

Moreover, $s=0$ if and only if the image of $f$ is contained in the zero section.
\end{proposition}

\begin{proof}
 Since $X$ is smooth, $Y$ is normal and $p$ is finite, there exists an open set $X^\circ\subset X$ such that $Y^\circ\coloneqq p^{-1}(X^\circ)$ is smooth, $f|_{Y^\circ}$ is a closed embedding, $p^\circ\colon Y^\circ\to X^\circ$ is finite and flat and the complements of $X^\circ$ and $Y^\circ$ have codimension at least two. Set $\lambda_Y=f^*\lambda_{\mathrm{taut}}\in H^0(Y,p^*\Omega_X^1)$. 

Firstly we show that $f$ determines a generically regular spectral datum $s\in S_X^r$.  Multiplication by $\lambda_Y$ gives an $\mathscr{O}_{Y^\circ}$-linear map
\begin{equation}
	\label{eq.times-lambda}
	\mathscr{O}_{Y^\circ}\longrightarrow  \mathscr{O}_{Y^\circ}\otimes(p^\circ)^*\Omega_{X^\circ}^1.
\end{equation}
Set $\mathscr{E}^{\circ}=p^{\circ}_*\mathscr{O}_{Y^{\circ}}$. Then $\mathscr{E}^{\circ}$ is locally free because $p^{\circ}$ is flat. In particular, after pushing forward \eqref{eq.times-lambda} and using the projection formula, we obtain a Higgs bundle
\[
\theta^{\circ}\colon \mathscr{E}^{\circ} \longrightarrow \mathscr{E}^{\circ}\otimes \Omega_{X^{\circ}}^1.
\]
Then the image $s^\circ=h_{X^{\circ}}^r(\mathscr{E}^{\circ},\theta^{\circ})\in S_{X^{\circ}}^r$ is generically regular. On the other hand, as $\operatorname{codim}(X\setminus X^{\circ})\geq 2$, by Hartogs' theorem, the spectral datum $s^{\circ}$ extends uniquely to a spectral datum $s\in S_X^r$. Moreover, we observe that $s=0$ if and only if $\lambda_{Y}|_{Y^{\circ}}=0$, and therefore if and only if $f(Y)$ is contained in the zero section of $T^*X$.

Next we construct the injection from $\mathcal{M}_{\cm}(Y)$ to $\mathcal{H}_X^r$. Let $\mathscr{L}\in \mathcal{M}_{\cm}(Y)$. Again multiplication by $\lambda_Y$ yields an $\mathscr{O}_Y$-linear map
\[
  \mathscr{L}\longrightarrow \mathscr{L}\otimes p^*\Omega_X^1.
\]
Set $\mathscr{E}=p_*\mathscr{L}$. Since $p$ is finite and $X$ is smooth, the sheaf $\mathscr{E}$ is Cohen--Macaulay and hence is locally free by \cite[IV, B), Proposition 11 and D), Corollaire 2]{Serre1965}. Then pushing forward and using the projection formula give a Higgs bundle
\[
  \theta\colon\mathscr{E}\coloneqq p_*\mathscr{L}
  \longrightarrow p_*(\mathscr{L}\otimes p^*\Omega_X^1)=\mathscr E\otimes\Omega_X^1.
\]
Note that the spectral datum $h_X^r(\mathscr{E},\theta)$ depends only on $\lambda_{Y}$, so it coincides with the spectral datum $s$ over $X^{\circ}$. Hence $h_X^r(\mathscr{E},\theta)=s$ and $(\mathscr{E},\theta)\in \mathcal{H}_X^r(s)$. 

Finally we construct the inverse map from $\mathcal{H}_X^r(s)$ to $\mathcal{M}_{\cm}(Y)$. Let $(\mathscr{E},\theta)\in\mathcal H_X^r(s)$. Then the commuting local endomorphisms of $\mathscr{E}$ defined by $\theta$ induce a natural $\mathscr{O}_{T^*X^{\circ}}$-module structure on $\mathscr{E}^{\circ}\coloneqq \mathscr{E}|_{X^{\circ}}$. Moreover, the Cayley--Hamilton theorem then implies that the $\mathscr{O}_{T^*X^{\circ}}$-module structure on $\mathscr{E}^{\circ}$ actually factors through the quotient $\mathscr{O}_{T^*X^{\circ}}\rightarrow \mathscr{O}_{Y^{\circ}}$. However, since $Y$ is normal, $\mathscr{E}$ is locally free, and the complements of $X^{\circ}$ and $Y^{\circ}$ have codimension at least two, there exist natural isomorphisms
\[
\bigl(\iota_{Y^{\circ}}\bigr)_*\mathscr{O}_{Y^{\circ}}\cong \mathscr{O}_{Y}\quad \text{and}\quad \bigl(\iota_{X^{\circ}}\bigr)_*\mathscr{E}^{\circ}\cong \mathscr{E},
\]
where $\iota_{Y^{\circ}}\colon Y^{\circ}\rightarrow Y$ and $\iota_{X^{\circ}}\colon X^{\circ}\rightarrow X$ are the natural inclusions. This implies that the $\mathscr{O}_{Y^{\circ}}$-module structure on $\mathscr{E}^{\circ}$ extends to a unique $\mathscr{O}_Y$-module structure on $\mathscr{E}$. Moreover, since $\mathscr{E}$ is Cohen--Macaulay and $p$ is finite, it follows that $\mathscr{E}$ is also Cohen--Macaulay as an $\mathscr{O}_Y$-module and hence it defines a rank one torsion free sheaf on $Y$, denoted by $\mathscr{L}$. On the other hand, since $\mathscr{E}$ is locally free and $X$ is smooth, the sheaf $\mathscr{E}$ is Cohen--Macaulay as an $\mathscr{O}_X$-module. In particular, since $p$ is finite, the sheaf $\mathscr{L}$ is Cohen--Macaulay as an $\mathscr{O}_Y$-module by \cite[IV, B), Proposition 11]{Serre1965}, and consequently $\mathscr{L}\in \mathcal{M}_{\cm}(Y)$. Now one can easily check that this is the inverse map of $\mathcal{M}_{\cm}(Y)\to \mathcal{H}_X^r$ given in the previous step.
\end{proof}

\subsection{Criterion for regularity}

To compactify our local counterexamples to projective counterexamples, we need the following technical lemma.

\begin{lemma}
	\label{lem:Ramified-Index}
	Let $\gamma\colon Y_0\longrightarrow X_0$ be a finite morphism from a normal complex variety $Y_0$ to a smooth variety $X_0$, and let $\tau\colon X_1\longrightarrow X_0$	be a finite morphism, where $X_1$ is smooth. Let $Y_1$ be the normalization of an irreducible component of
	$X_1\times_{X_0}Y_0$, and consider the commutative diagram
	\[
	\begin{tikzcd}[column sep=large,row sep=large]
		\mathfrak D_1\subset Y_1
		\arrow[r,"{\tau'}"]
		\arrow[d,"{\gamma'}" left]
		&
		Y_0\supset \mathfrak D_0
		\arrow[d,"\gamma"]
		\\
		\mathfrak B_1\subset X_1
		\arrow[r,"\tau"]
		&
		X_0\supset \mathfrak B_0.
	\end{tikzcd}
	\]
	Let $\mathfrak B_0$ be a prime divisor on $X_0$, let $\mathfrak B_1$ be a prime divisor on $X_1$ lying over	$\mathfrak B_0$ with ramification index $m$, and let
	$\mathfrak D_0$ be a prime divisor on $Y_0$ lying over $\mathfrak B_0$ with ramification index $e$. Assume that $e\mid m$. Then, for every prime divisor $\mathfrak D_1$ lying over both	$\mathfrak B_1$ and $\mathfrak D_0$, and every rational section $\sigma$ of $\gamma^*\Omega_{X_0}^1$ such that $(d\gamma)(\sigma)$ is regular at the generic point of $\mathfrak D_0$, the rational section $\Psi(\tau'^*\sigma) $	is regular at the generic point of $\mathfrak D_1$, where
	\[
	\Psi\colon
	\tau'^*\gamma^*\Omega_{X_0}^1
	\xlongrightarrow{\ \cong\ }
	\gamma'^*\tau^*\Omega_{X_0}^1
	\xlongrightarrow{\ \gamma'^*(d\tau)\ }
	\gamma'^*\Omega_{X_1}^1.
	\]
\end{lemma}

\begin{proof}
	Set $n=\dim X_0$. Let $t_0$ be a uniformizer at $\mathfrak B_0$, and choose $y_1,\ldots,y_{n-1}\in K(X_0)$ such	that $dt_0,dy_1,\ldots,dy_{n-1}$ form a basis of \(\Omega_{X_0}^1\) at the generic point of $\mathfrak B_0$. Write
	\[
	\sigma=f_0\,dt_0+\sum_{j=1}^{n-1}f_j\,dy_j,
	\qquad f_j\in K(Y_0).
	\]
	A straightforward computation then shows that the regularity of $(d\gamma)(\sigma)$	implies
	\begin{equation}
		\label{eq.Reg-order}
		\ord_{\mathfrak D_0}(f_0)\ge 1-e,
		\qquad
		\ord_{\mathfrak D_0}(f_j)\ge0.
	\end{equation}
	Similarly, if $t_1$ is a uniformizer at $\mathfrak B_1$, then $\tau^*t_0=b t_1^m$ for a unit $b$. Moreover, the elements $dt_1,d(\tau^*y_1),\dots,d(\tau^*y_{n-1})$ form a basis of $\Omega_{X_1}^1$ at the generic point of $\mathfrak{B}_1$ such that
	\[
	d(\tau^*t_0)
	=
	t_1^{m-1}v\,dt_1
	+
	t_1^m\sum_{j=1}^{n-1}b_j\,d(\tau^*y_j),
	\]
	with \(v\) a unit and the \(b_j\) regular. Put \(q=m/e\). By Abhyankar's lemma,
	\[
	e(\mathfrak D_1/\mathfrak B_1)=1,
	\qquad
	e(\mathfrak D_1/\mathfrak D_0)=q,
	\]
	which implies
	\begin{equation}
		\ord_{\mathfrak D_1}(t_1)=1 \quad \text{and}\quad \label{eq.Ab-RI}
		\ord_{\mathfrak D_1}(\tau'^*g)
		=q\,\ord_{\mathfrak D_0}(g),
		\qquad g\in K(Y_0)^\times.
	\end{equation}
	Moreover, we have
	\[
	\Psi(\tau'^*\sigma)
	=(\tau'^*f_0)t_1^{m-1}v\,dt_1\\
	+\sum_{j=1}^{n-1}
	\bigl(\tau'^*f_j+(\tau'^*f_0)t_1^mb_j\bigr)
	d(\tau^*y_j).
	\]
	Then one can easily derive from \eqref{eq.Reg-order} and \eqref{eq.Ab-RI} that $\Psi(\tau'^*\sigma)$ is regular at the generic point of $\mathfrak{D}_1$. 
\end{proof}

\section{Local examples}
\label{ss.Local-Models}

By the spectral correspondence established in Proposition~\ref{prop:general-higgs-factors-through-Z}, the nonexistence of rank one MCM modules on a spectral variety provides a natural local obstruction to the surjectivity of the spectral morphism. In this section, we show that any normal affine variety $Y$ admitting no rank one MCM modules can be realized as such a local obstruction. We conclude by constructing explicit examples of such varieties, following~\cite{Ma2019}.

\subsection{From finite covers to spectral varieties}

Let $Y$ be a normal affine variety of dimension $n$.  Choose a Noether
normalization $p\colon Y\longrightarrow\A^n$ of degree $r$.  Write
\[
  A(Y)=\Gamma(Y,\OO_Y),
  \qquad
  A(\A^n)=\C[x_1,\ldots,x_n].
\]
Choose elements $g_1,\ldots,g_m\in A(Y)$ such that
$g_1,\ldots,g_m$ generate $A(Y)$ as an $A(\A^n)$-algebra.
They define a
closed embedding
\[
  \begin{aligned}
  \iota_Y\colon Y&\lhook\joinrel\longrightarrow\A^{n+m}\ni (x,u),\\
  y&\longmapsto\bigl(p(y),g_1(y),\ldots,g_m(y)\bigr).
  \end{aligned}
\]
Since $\A^n$ is factorial, the reduced divisorial part of the branch locus of
$p$ is defined by a square-free polynomial
$B\in\C[x_1,\ldots,x_n]$, unique up to a non-zero scalar.  Put
$W=Y\times\A^m$, with coordinates $z=(z_1,\ldots,z_m)$ on the second factor, and
consider the finite morphism
\[
  \begin{aligned}
  q=p\times\operatorname{Id}_{\A^m}\colon
  W&\longrightarrow\A^n\times\A^m,\\
  (y,z)&\longmapsto\bigl(p(y),z\bigr).
  \end{aligned}
\]
Define
\begin{equation}\label{eq:def-F}
  F\coloneqq B\sum_{j=1}^m g_jz_j\in A(W).
\end{equation}
For $1\leq i\leq n$, the derivation
$\partial_i=\partial/\partial x_i$ of $\C(x_1,\ldots,x_n,z_1,\ldots,z_m)$
extends uniquely to a derivation
\[
  D_i\colon K(W)\longrightarrow K(W).
\]

\begin{lemma}\label{lem:der-reg}
For every $H\in A(Y)$ and any codimension one point $P\in Y$, we have $\ord_P D_i(H)\geq 1-e$ for every $1\leq i\leq n$, where $e$ is the ramification index of $p$ at $P$.  

In particular, we have $BD_i(H)\in A(Y)$ for every $1\leq i\leq n$.
\end{lemma}

\begin{proof}
Since $Y$ is normal, it is enough to check the regularity of $BD_i(H)$ at every codimension-one point $P\in Y$.  Let $Q=p(P)$ and let $e$ be the ramification index of $P$ over $Q$.  By the definition of $B$,
\[
  \ord_P(B)\geq e-1.
\]
It therefore suffices to prove
$\ord_PD_i(H)\geq1-e$. Set
\[
  a\coloneqq
  \min_{1\leq j\leq m}\ord_PD_i(g_j).
\]
Since $D_i(A(\mathbb{C}^n))\subset A(\mathbb{C}^n)$, and $A(Y)$ is generated by $1$ and the $g_j$'s as an $A(\mathbb{C}^n)$-algebra, the chain rule shows that it is enough to prove that $a\geq1-e$.  Suppose, to the
contrary, that $a\leq -e$.  Choose $k$
with $\ord_PD_i(g_k)=a$. Let $t$ be a uniformizer of the discrete valuation
ring $\mathscr{O}_{Y,P}$, and let $\delta=t^{-a}D_i$. Then $\delta(\mathscr{O}_{Y,P})\subset\mathscr{O}_{Y,P}$ and $\delta(g_k)$ is a unit in $\mathscr{O}_{Y,P}$. Let $s$ be a uniformizer of $\mathscr{O}_{\A^n,Q}$. Then we have
\begin{equation}\label{eq:ramification-index}
  p^*s=ut^e
\end{equation}
for some unit $u\in\mathscr{O}_{Y,P}$.  Since $D_i$ preserves
$\mathscr{O}_{\A^n,Q}$, we have
\[
  \ord_P\delta(p^*s)=\ord_P t^{-a} + \ord_P D_i(p^*s) \geq-a\geq e.
\]
Applying $\delta$ to \eqref{eq:ramification-index} gives
\[
  t^e\mathscr{O}_{Y,P}\ni \delta(p^*s)=\delta(u)t^e+eut^{e-1}\delta(t),
\]
and hence $\delta(t)\in t\mathscr{O}_{Y,P}$.  Thus $\delta$ induces a derivation
$\bar\delta$ of the residue field $\kappa(P)$.  Moreover,
\[
  \delta(\mathscr{O}_{\A^n,Q})=t^{-a}D_i(\mathscr{O}_{\mathbb{A}^n,Q})\subset t^{-a}\mathscr{O}_{Y,P}
  \subset t\mathscr{O}_{Y,P},
\]
so $\bar\delta$ vanishes on $\kappa(Q)$.  The extension
$\kappa(P)/\kappa(Q)$ is finite and separable; therefore $\bar\delta=0$.  It
follows that $\delta(\mathscr{O}_{Y,P})\subset t\mathscr{O}_{Y,P}$, contradicting the fact that $\delta(g_k)$ is a unit.
\end{proof}

\begin{lemma}
\label{lem:dBD_K(g)}
    Let $P$ be a codimension one point of $Y$ such that $Q=p(P)\in \mathbb{A}^n$ is contained in the branch locus of $p$. Let $t$ be a uniformizer of $\mathscr{O}_{Y,P}$. Then for any $g\in \mathscr{O}_{Y,P}$, we have
    \[
    d(BD_k(g)) = \frac{D_k(B)}{e} dg \in \Omega_{Y/\mathbb{A}^n}^1\otimes \kappa(P).
    \]
\end{lemma}

\begin{proof}
    By Lemma \ref{lem:der-reg} (and its proof), the derivation $BD_k$ acts as zero on $\kappa(P)$, i.e. $BD_k(\mathscr{O}_{Y,P})\subset t\mathscr{O}_{Y,P}$. Moreover, without loss of generality, we may assume that $\Omega_{Y/\mathbb{A}^n}^1\otimes \kappa(P)\not=0$ with $dt$ a basis; that is, the ramification index $e$ of $p$ at $P$ is at least two. Then, by Lemma \ref{lem:der-reg}, we have $BD_k(\mathscr{O}_{\mathbb{A}^n,Q})\subset t^2 \mathscr{O}_{Y,P}$ and $BD_k(\mathscr{O}_{Y,P})\subset t\mathscr{O}_{Y,P}$. It therefore follows that the induced morphism 
    \[
    \overline{BD}_k\colon \mathscr{O}_{Y,P}\xlongrightarrow{BD_k} t\mathscr{O}_{Y,P}\xlongrightarrow{d}  t\mathscr{O}_{Y,P}/t^2\mathscr{O}_{Y,P}\cong \Omega^1_{Y/\mathbb{A}^n}\otimes \kappa(P)
    \]
    is an $\mathscr{O}_{\mathbb{A}^n,Q}$-derivation. In particular, by the universal property of differentials, there exists a unique morphism $\delta\colon \Omega^1_{Y/\mathbb{A}^n}\otimes \kappa(P)\rightarrow \Omega^1_{Y/\mathbb{A}^n}\otimes \kappa(P)$ of $\kappa(P)$-vector spaces such that $\overline{BD}_k=\delta\circ d$. On the other hand, note that $B$ is a uniformizer of $\mathscr{O}_{\mathbb{A}^n,Q}$, it then follows from \eqref{eq:ramification-index} that
    \[
    BD_k(t) = \frac{D_k(B)}{e}t \mod \langle t^2\rangle \quad \bigl(\in \,\mathscr{O}_{Y,P}\bigr).
    \]
    This implies
    \[
    \delta(dt) = \frac{D_k(B)}{e} dt  \,\in\Omega_{Y/\mathbb{A}^n}^1\otimes \kappa(P).
    \]
    Then the results follow because $d\circ (BD_k)=\overline{BD}_k=\delta\circ d$ and $\delta$ is $\kappa(P)$-linear.
\end{proof}

Lemma~\ref{lem:der-reg} shows that each $D_i(F)$ is regular on $W$.  We 
therefore define
\begin{equation}\label{eq:morphism-f}
  \begin{aligned}
  f\colon W&\longrightarrow\A^{2n+2m}\ni (x,z,x',z'),\\
  (y,z)&\longmapsto
  \bigl(p(y),z,D_1(F),\ldots,D_n(F),
        Bg_1,\ldots,Bg_m\bigr).
  \end{aligned}
\end{equation}

\begin{remark}
\label{rem.Geom-f}
The morphism \eqref{eq:morphism-f} can be described as follows. Any rational function $H\in K(W)$ determines a rational section of the vector bundle $q^*\Omega_{\mathbb A^{n+m}}^1\to W$ given by
\[
\sigma_H
\coloneqq
\sum_{i=1}^n D_i(H)\,dx_i
+
\sum_{j=1}^m \frac{\partial H}{\partial z_j}\,dz_j.
\]
Lemma~\ref{lem:der-reg} shows that $\sigma_F$ is in fact a regular section. Identifying $\mathbb A^{2n+2m}$, with coordinates $(x,z,x',z')$, with $T^*\mathbb A^{n+m}$, where $x_i'$ and $z_j'$ correspond to $dx_i$ and $dz_j$, respectively, the morphism $f$ is precisely the composition
\[
W
\xlongrightarrow{\ \sigma_F\ }
q^*T^*\mathbb A^{n+m}
\longrightarrow
T^*\mathbb A^{n+m}
=
\mathbb A^{2n+2m},
\]
and $f^*\lambda_{\operatorname{taut}}=dq(\sigma_F)=dF$. In particular, our next result shows that the map $f\colon W\rightarrow T^*\mathbb{A}^{n+m}$ is an \emph{exact Lagrangian embedding} in codimension one.
\end{remark}

\begin{proposition}\label{prop:f-codim-one}
There is an open
subset $W^\circ\subset W$, with
$\operatorname{codim}(W\setminus W^\circ)\geq2$, such that
$f|_{W^\circ}$ is an embedding.
\end{proposition}

\begin{proof}
	Consider the commutative diagram
	\[
	\begin{tikzcd}[row sep=large, column sep=large]
		(y,z)\in W
		\arrow[r,"f"]
		\arrow[d,hook,"\iota_Y\times\operatorname{Id}_{\mathbb{A}^m}"']
		\arrow[dr,"q"]
		&
		\mathbb{A}^{2n+2m}\ni(x,z,x',z')
		\arrow[d,"\operatorname{pr}_{x,z}"]
		\\
		(x,u,z)\in \mathbb{A}^{n+2m}
		\arrow[r,"\operatorname{pr}_{x,z}"']
		&
		\mathbb{A}^{n+m}\ni (x,z).
	\end{tikzcd}
	\]
	Over the open set $B\neq 0$, notice that the coordinates
$u_j=g_j(y)$ can be recovered from $f(y,z)$ by the formula
\[
u_j=\frac{z'_j}{B(x)}.
\]
Thus $f$ is a closed embedding over $B\neq 0$ and it only remains to consider the points contained in the closed subset $B=0$.

Firstly, we consider the injectivity of $f$. Let $\sB$ be an irreducible component of the zero locus $V(B)$ of $B$ in $\mathbb{A}^n$.  Since $B$ is square-free,
there is an index $k$ such that $\partial B/\partial x_k$ does not vanish
identically on $\sB$.  For a point $w=(y,z)\in q^{-1}(\sB\times\A^m)$, by Lemma \ref{lem:der-reg}, we have
\begin{equation}\label{eq:DkF-branch}
  D_k(F)(w)=D_k(B)(p(y))\sum_{j=1}^m g_j(y)z_j(w).
\end{equation}
Fix a point $x\in \mathfrak{B}$ such that $D_k(B)(x)\not=0$. Write the fiber $p^{-1}(x)$ as $\{y_1,\ldots,y_\ell\}$. Then the vectors
\[
  \bigl(g_1(y_a),\ldots,g_m(y_a)\bigr)
  \qquad(1\leq a\leq\ell)
\]
are pairwise distinct because $\iota_Y$ is an embedding. Then for a general point $z=(z_1,\dots,z_m)$, the following values
\[
\sum_{j=1}^m g_j(y_a) z_j\quad (1\leq a\leq \ell)
\]
are again pairwise distinct. As a consequence, there exists an open subset $U$ of $q^{-1}(\mathfrak{B}\times \mathbb{A}^m)$ such that for any point $w=(y,z)\in U$, we have $D_k(B)(p(y))\not=0$ and $D_k(F)$ separates the points contained in the fiber $q^{-1}(q(w))$. Hence $f$ is injective in codimension one.

Next we consider the injectivity of the differentials of $f$. Let $\mathfrak{D}$ be an irreducible component of $q^{-1}(\mathfrak{B}\times \mathbb{A}^m)$, dominating $\mathfrak{B}$. 
Let $t$ be a uniformizer of $\mathscr{O}_{W,\eta}$, where $\eta$ is the generic point of $\mathfrak{D}$. Assume that $D_k(B)\not= 0$ in the residue field of the generic point of $\mathfrak{B}$. Note that there exists $g_s$ such that $dg_s$ is a basis of $\Omega_{W/\mathbb{A}^{n+m},\eta}^1\otimes \kappa(\eta)$ over $\kappa(\eta)$. Then, by Lemmas \ref{lem:der-reg} and \ref{lem:dBD_K(g)}, we have
\begin{align*}
0\not=f^* dx'_k = d(D_k(F)) 
    & = D_k(B) \sum_{j=1}^m z_j dg_j + \sum_{j=1}^m z_j d(BD_k(g_j)) \\
    & = \frac{e+1}{e} D_k(B)\sum_{j=1}^m z_j dg_j\in \Omega^1_{W/\mathbb{A}^{n+m},\eta}\otimes \kappa(\eta).
\end{align*}
It therefore follows from Nakayama's lemma that there exists a dense open subset $\mathfrak{D}^{\circ}$ of $\mathfrak{D}$ such that the differential of $f$ along $\mathfrak{D}^{\circ}$ is injective.  Since $p$ is finite, it follows that the differential of $f$ is injective in codimension one and hence $f$ is an embedding in codimension one by the injectivity established in the previous step, as required.

\end{proof}

\begin{corollary}\label{cor:local-model}
The morphism $f$ determines a non-zero generically regular spectral datum $s\in S_{\A^{n+m}}^r$ such that $\mathcal H_{\A^{n+m}}^r(s)$ is canonically identified with $\mathcal{M}_{\cm}(Y)$.
\end{corollary}

\begin{proof}
Identify $\A^{2n+2m}$, with coordinates $(x,z,x',z')$, with
$T^*\A^{n+m}$ by letting $x_i'$ and $z_j'$ correspond to $dx_i$ and $dz_j$,
respectively.  The composition of $f$ with the cotangent projection is $q$.
Propositions~\ref{prop:general-higgs-factors-through-Z} and
\ref{prop:f-codim-one} therefore give a generically regular datum $s$ for
which $\mathcal H_{\A^{n+m}}^r(s)$ is identified with $\mathcal{M}_{\cm}(W)$. 

Note that there exists a natural isomorphism $\Cl(W)\cong \Cl(Y)$ by pulling back along the projection $p_Y\colon W=Y\times \mathbb{A}^m\rightarrow Y$, so it remains to prove that for a divisor $D$ on $Y$, the $\mathscr{O}_W$-module $\mathscr{O}_W(p_Y^*D)$ is Cohen--Macaulay if and only if the $\mathscr{O}_Y$-module $\mathscr{O}_Y(D)$ is Cohen--Macaulay. Since $\mathscr{O}_Y(D)$ is reflexive and $p_Y$ is flat, the pullback $p_Y^*\mathscr{O}_Y(D)$ is also reflexive \cite[Proposition 1.8]{Hartshorne1980}, which then implies $\mathscr{O}_W(p_Y^* D)\cong p_Y^*\mathscr{O}_Y(D)$, because they agree over the big open subset $Y_{\operatorname{reg}}\times \mathbb{A}^m$ of $W$. Here we note that $p_Y^*D$ is well-defined by restricting to the generic points of the support of $D$. Finally, since $p_Y$ is smooth, by the depth formula \cite[Lemma 10.163.1]{StacksProject}, the pullback $p^*_Y\mathscr{O}_Y(D)\cong \mathscr{O}_W(p_Y^*D)$ is Cohen--Macaulay if and only if $\mathscr{O}_Y(D)$ is Cohen--Macaulay.
\end{proof}

\begin{remark}
    In the local setting, it is not necessary to work with the morphism $f\colon W\longrightarrow T^*\mathbb A^{n+m}$    induced by the function $F$ as in Remark~\ref{rem.Geom-f}. Instead, one may take any regular section $\sigma\in H^0\bigl(W,q^*\Omega_{\mathbb A^{n+m}}^1\bigr)$, for instance $\sigma=\sum_j g_j\,dz_j$, such that the induced morphism
    \[
    W\xlongrightarrow{\ \sigma\ } q^*T^*\mathbb A^{n+m}
    \longrightarrow T^*\mathbb A^{n+m}
    \]
    is an embedding in codimension one. However, for our later application to compactify the local model to a global one, we require that $\sigma$ arises from a regular function on $W$, such as $F$. 
\end{remark}

\subsection{Cones without rank one MCM sheaves}
\label{sec:cones}

We now construct normal affine varieties which admit no rank one MCM sheaves, following the idea from \cite{Ma2019}. Let $(Z,\mathscr{O}_Z(1))$ be a \emph{projectively normal} polarized smooth projective variety of dimension $n-1\geq2$; in other words, the section ring
\[
R(Z,\mathscr{O}_Z(1))\coloneqq \bigoplus_{q\geq 0} H^0(Z,\mathscr{O}_Z(q))
\]
is generated in degree one.  Then the \emph{cone} over $(Z,\mathscr{O}_Z(1))$ is defined to be the normal affine variety
\[
  Y\coloneqq
  \Spec\!R(Z,\mathscr{O}_Z(1)).
\]
The natural grading defines a $\mathbb{C}^{\times}$-action on $Y$ with a unique fixed point $o$, called the \emph{vertex}. Then $Y$ is normal with the vertex $o$ being its unique singular point. The punctured cone carries a
natural $\C^\times$-bundle
\[
  p_Y\colon Y\setminus\{o\}\longrightarrow Z,
\]
and pullback induces an isomorphism
\begin{equation}
	\label{eq.Isom-Cl}
	  \Pic(Z)/\mathbb Z[\mathscr{O}_Z(1)]
	\xrightarrow{\ \cong\ }
	\Pic(Y\setminus\{o\})
	\cong\Cl(Y).
\end{equation}
Recall that a line bundle $\mathscr L$ on $Z$ is \emph{arithmetically
Cohen--Macaulay} (\emph{aCM} for short) with respect to $\mathscr{O}_Z(1)$ if $  H^i(Z,\mathscr L(q))=0$ for every $q\in\mathbb Z$ and every $0<i<\dim Z=n-1$. The following result relates the existence of rank one MCM sheaves on $Y$ to the existence of aCM line bundles on $(Z,\mathscr{O}_Z(1))$.

\begin{proposition}\label{prop:cone-acm}
Let $(Z,\mathscr{O}_Z(1))$ be a projectively normal polarized projective manifold with $Y$ its cone. Then the following conditions are equivalent.
\begin{enumerate}
\item\label{item:cone-global-mcm}
The cone $Y$ admits a rank one MCM sheaf;
\item\label{item:cone-local-mcm}
The local ring $\OO_{Y,o}$ admits a rank one MCM module;
\item\label{item:base-acm}
The polarized projective manifold $(Z,\OO_Z(1))$ admits an aCM line bundle.
\end{enumerate}
\end{proposition}

\begin{proof}
The implication
\ref{item:cone-global-mcm}$\Rightarrow$\ref{item:cone-local-mcm} follows by
taking the stalk at $o$.  Conversely, since $Y$ is normal, Nagata's theorem says that the localization at $o$ induces a surjection
\[
  \Cl(Y)\longrightarrow\Cl(\Spec\mathscr{O}_{Y,o}).
\]
Thus a rank one MCM $\mathscr{O}_{Y,o}$-module extends to a
rank one reflexive sheaf $\mathscr{L}$ on $Y$ such that $\mathscr{L}$ is Cohen--Macaulay at $o$. Moreover, since $Y$ is smooth away from $o$, it follows that $\mathscr{L}$ is locally free away from $o$, and hence $\mathscr{L}$ is Cohen--Macaulay on $Y$. This proves the equivalence of
\ref{item:cone-global-mcm} and \ref{item:cone-local-mcm}.

Now we prove the equivalence between \ref{item:cone-global-mcm} and \ref{item:base-acm}. According to the isomorphism \eqref{eq.Isom-Cl}, every rank one reflexive sheaf on $Y$ is isomorphic to some $\iota_*p_Y^*\mathscr{L}$, where $\iota\colon Y\setminus \{o\}\rightarrow Y$ and $\mathscr{L}$ is a line bundle, unique up to a twist by $\mathscr{O}_Z(1)$, on $(Z,\mathscr{O}_Z(1))$. Note that $\iota_*p_Y^*\mathscr{L}$ is the sheafification of the following graded $A(Y)$-module:
\[
M(\mathscr{L})\coloneqq \bigoplus_{q\in \mathbb{Z}} H^0(Z,\mathscr{L}(q)).
\]
Then it is known that $\mathscr{L}$ is an aCM line bundle if and only if $M(\mathscr{L})$ is Cohen--Macaulay (see \cite[Proposition 2.2.4]{KleimanLandolfi1971}).
\end{proof}

\begin{example}\label{ex:bidegree}
	Let $n\geq3$, $a\geq 1$, and $b\geq a+n+2$. Assume furthermore that $a\geq 2$ if $n=3$. Let $A=\mathbb P^1\times\mathbb P^{n-1}$. For any $(a,b)\in \mathbb{Z}^2$, write $\mathscr{O}_A(a,b)=\mathscr{O}_{\mathbb{P}^1}(a)\boxtimes \mathscr{O}_{\mathbb{P}^{n-1}}(b)$. Let $Z\in |\mathscr{O}_A(a,b)|$ be a smooth hypersurface of bi-degree $(a,b)$ and assume also that $Z$ is very general if $n=3$. Let $Y$ be the cone over $(Z,\mathscr{O}_A(1,1)|_Z)$. Then $Y$ admits no rank one MCM sheaves by Proposition \ref{prop:cone-acm} and Lemma \ref{lem.Ex} below. 
\end{example}

\begin{lemma}
	\label{lem.Ex}
	The polarized projective manifold $(Z,\mathscr{O}_A(1,1)|_Z)$ in Example \ref{ex:bidegree} admits no aCM line bundles.
\end{lemma}

\begin{proof}
	For simplicity, we write $\mathscr{O}_Z(u,v)=\mathscr{O}_A(u,v)|_Z$ for any $(u,v)\in \mathbb{Z}^2$. Note that $(A,\mathscr{O}_A(1,1))$ is projectively normal and the restriction $H^0(A,\mathscr{O}_A(q,q))\rightarrow H^0(Z,\mathscr{O}_Z(q,q))$ is surjective for any $q\in \mathbb{Z}$ by Kodaira's vanishing theorem, so $(Z,\mathscr{O}_Z(1,1))$ is also projectively normal. 
	
	Moreover, we also notice that the restriction map $\Pic(A)\rightarrow\Pic(Z)$
	is an isomorphism.  For $n\geq4$ this follows from the
	Grothendieck--Lefschetz theorem.  For $n=3$ it follows from the
	Noether--Lefschetz theorem, in the form of
	\cite[Theorem 1]{RavindraSrinivas2009}, since the adjoint line bundle
	\[
	\omega_{\mathbb{P}^1\times \mathbb{P}^2}\otimes\bigl(\mathscr{O}_{\mathbb{P}^1}(a)\boxtimes \mathscr{O}_{\mathbb{P}^2}(b)\bigr)
	\cong \mathscr{O}_{\mathbb{P}^1}(a-2)\boxtimes \mathscr{O}_{\mathbb{P}^2}(b-3)
	\]
	is globally generated for $a\geq 2$ and $b\geq a+4>3$. On the other hand, note that twisting by $\mathscr{O}_Z(1,1)$ preserves the aCM property. Thus for any line bundle $\mathscr{L}\in \Pic(Z)$,  after a twist by suitable $\mathscr{O}_{Z}(q,q)$, we may assume that $\mathscr{L}\cong \mathscr L_c\coloneqq\mathscr{O}_Z(0,c)$ for some $c\in \mathbb{Z}$. 
	
	Firstly, we assume that $c\geq2$ and set $t=-c$.  Then
	$\mathscr L_c(t)=\mathscr{O}_Z(-c,0)$.  Consider the following exact sequence:
	\[
	0\longrightarrow\mathscr{O}_A(-a-c,-b)
	\longrightarrow\mathscr{O}_A(-c,0)
	\longrightarrow\mathscr{O}_Z(-c,0)
	\longrightarrow0.
	\]
	By the K\"unneth formula, we have 
	\[	
	H^1(A,\mathscr{O}_A(-c,0))\neq 0\quad \text{and}\quad
	H^1(A,\mathscr{O}_A(-a-c,-b))
	=H^2(A,\mathscr{O}_A(-a-c,-b))=0.
	\]
	This implies $H^1(Z,\mathscr L_c(-c))\neq0$ and hence $\mathscr L_c$ is not aCM.
	
	Next we assume that $c\leq1$ and set $t\coloneqq\max\{a,-c-n+1\}$. Consider the following exact sequence:
	\[
	0\longrightarrow\mathscr{O}_A(t-a,c+t-b)
	\longrightarrow\mathscr{O}_A(t,c+t)
	\longrightarrow\mathscr{O}_Z(t,c+t)
	\longrightarrow0.
	\]
	We claim that
	\[
	H^{n-1}\!\left(A,\mathscr{O}_A(t-a,c+t-b)\right)\neq0.
	\]
	Indeed, if $t=a$, then
	\[
	c+t-b=c+a-b\leq c-n-2\leq-n-1,
	\]
	while, if $t=-c-n+1$, then
	\[
	c+t-b=1-n-b\leq-n.
	\]
	In particular, we always have $H^{n-1}(\mathbb{P}^{n-1},\mathscr{O}_{\mathbb{P}^{n-1}}(c+t-b))\not=0$, and the claim follows from the K\"unneth formula as $t-a\geq 0$.
	
	On the other hand, $t\geq0$ and $c+t\geq1-n$, so
	\[
	H^{n-2}(A,\mathscr{O}_A(t,c+t))
	=H^{n-1}(A,\mathscr{O}_A(t,c+t))=0.
	\]
	This yields $H^{n-2}(Z,\mathscr L_c(t))\neq0$ and thus $\mathscr L_c$ is not aCM in this case either.
\end{proof}

\begin{example}[{\cite[Theorem 3.4]{Ma2019}}]\label{ex:ma}
Let $Z\subset\mathbb P^N$ be a general smooth complete intersection of
multi-degree $(d_1,\ldots,d_{N-n+2})$ such that $\sum d_i > N+1$ and $n\geq 5$. Then the cone $Y$ over the polarized projective manifold 
$(Z\times \mathbb{P}^1, \mathscr{O}_Z(1)\boxtimes \mathscr{O}_{\mathbb{P}^1}(1))$ admits no rank one MCM sheaf. 
\end{example}

\section{Projective examples}
\label{sec.global-counterexamples}

In this section, we introduce a general framework to compactify the local counterexamples to global ones. In particular, applying it to the local counterexamples from the last section yields projective counterexamples to Chen--Ng\^o's conjecture, therefore proving Theorem \ref{thm:Chen-Ngo-false}.

\subsection{Local setup}

Throughout this section, let $(Y,o_Y)$ be a pointed normal affine variety of dimension $n\geq 3$ such that
\[
\tag{$\clubsuit$}
\widehat{\mathscr O}_{Y,o_Y}
\text{ admits no rank one MCM modules.}
\]
Let $p\colon Y\to\mathbb A^n$ be a Noether normalization of degree $r$, with $p(o_Y)=0$. We use the construction and notation of \S\,\ref{ss.Local-Models}. Set $W=Y\times\mathbb A^m$ and $q=p\times\operatorname{Id}_{\mathbb A^m}$. We then have a commutative diagram:
\begin{equation}
	\begin{tikzcd}[row sep=large,column sep=large]
		(y,z)\in Y\times \mathbb{A}^m = W
		\arrow[r,"f"]
		\arrow[d,"{\iota_Y\times\operatorname{Id}_{\mathbb A^m}}" left]
		\arrow[dr,"q"]
		& T^*\mathbb A^{n+m}=\mathbb A^{2n+2m}
		\ni(x,z,x',z')
		\arrow[d,"\operatorname{pr}_{x,z}"] \\
		(x,u,z)\in\mathbb A^{n+2m}
		\arrow[r,"\operatorname{pr}_{x,z}"]
		& \mathbb A^{n+m}\ni(x,z).
	\end{tikzcd}
\end{equation}
Recall from Remark \ref{rem.Geom-f} that $f$ is defined by the nonzero section $\sigma_F\in H^0\bigl(W,q^*\Omega^1_{\mathbb A^{n+m}}\bigr)$ associated with
\[
F=B\sum_{j=1}^m g_j z_j\in A(W).
\]
Here $B\in A(\mathbb A^n)$ is a square-free equation of the branch
divisor of $p$, and $1,g_1,\ldots,g_m$ generate $A(Y)$ as an $A(\mathbb A^n)$-algebra. By Proposition \ref{prop:f-codim-one}, the morphism $f$ is an embedding in codimension one. 

In the following, we aim to construct a finite cover $q_M\colon M\to\mathbb A^{n+m+1}$ and a rational section $\sigma_M$ of $q_M^*\Omega^1_{\mathbb A^{n+m+1}}$, starting with $q\colon W\rightarrow \mathbb{A}^{n+m}$ and $\sigma_F$,  satisfying the following properties:
\begin{itemize}
	\item The rational section $\sigma_M$ is regular at a point	$\mathfrak{o}\in M$ such that $\widehat{\mathscr O}_{M,\mathfrak{o}}$ admits no rank one MCM modules.
	
	\item The differential $(dq_M)(\sigma_M)$ is a regular $1$-form on $M$ which  extends regularly to the smooth locus
	of any normal projective birational model of $M$.	
\end{itemize}
Set $N=n+m+1$ and define $M=W\times E=Y\times \mathbb{A}^{m}\times E$, where $E$ is the affine elliptic curve
\[
E=\{v^2=u^3+u+1\}\subset\mathbb A^2\ni (u,v).
\]
We define 
\[
	q_M\colon M \longrightarrow\mathbb A^N=\mathbb{A}^{n+m+1}\ni (x,z,s),\quad (w,u,v)\longmapsto (x,z,s)\coloneqq \bigl(q(w),u+F(w)\bigr).
\]
Then $q_M$ is a finite morphism of degree $2r$.

\begin{lemma}
	\label{lem:Non-Exi-MCMmodules}
	For every point $o_M\in\{o_Y\}\times\mathbb A^m\times E$, the completion $\widehat{\mathscr O}_{M,o_M}$ admits no rank one MCM modules.
\end{lemma}

\begin{proof}
	Since $\mathbb A^m\times E$ is smooth of dimension $m+1$, there exists a (non-canonical) isomorphism
	\[
	\widehat{\mathscr O}_{M,o_M}
	\cong
	\widehat{\mathscr{O}}_{Y,o_Y}[[t_1,\ldots,t_{m+1}]].
	\]
	Let $\iota\colon Y'\coloneqq \Spec \widehat{\mathscr{O}}_{Y,o_Y}\rightarrow M'\coloneqq \Spec \widehat{\mathscr{O}}_{M,o_M}$ be the closed immersion induced by the natural quotient
	\[
	\widehat{\mathscr{O}}_{M,o_M}\longrightarrow \widehat{\mathscr{O}}_{M,o_M}/\langle t_1,\dots,t_{m+1}\rangle \cong \widehat{\mathscr{O}}_{Y,o_Y}.
	\]
	Let $I$ be a rank one MCM module over $\widehat{\mathscr{O}}_{M,o_M}$. 	Then $t_1,\ldots,t_{m+1}$ is an $I$-regular sequence, so the quotient 
	\[
	\overline{I}\coloneqq I/\langle t_1,\ldots,t_{m+1}\rangle I
	\]
	is MCM over $\widehat{\mathscr{O}}_{Y,o_Y}$. On the other hand, as $\iota(Y'_{\operatorname{reg}})\subset M'_{\operatorname{reg}}$, the quotient $\overline{I}$ also has rank one, which contradicts our assumption \((\clubsuit)\).
\end{proof}

Let $\overline E$ be the smooth projective completion of \(E\). Then the
differential $du/v$ extends to a nowhere-vanishing regular
differential on $\overline E$, denoted by $\xi\in H^0(\overline{E},\Omega_{\overline{E}}^1)$. Indeed, as $v^2=u^3+u+1$, we have
\[
\frac{du}{v} = \frac{2dv}{3u^2+1},
\]
so $du/v$ is regular around $v=0$. Moreover, at the infinity point $\{\infty\}=\overline{E}\setminus E$, a straightforward computation shows that $\ord_{\infty}(du)=\ord_{\infty}(v)=-3$, which then implies that $du/v$ is also regular around $\infty$. Let $\xi_M=\operatorname{pr}_E^*\xi\in H^0(M,\Omega^1_M)$. On the other hand, note that the equality $u=s-F\in A(M)$ defines a regular section
\[
	\sigma_u
	\coloneqq ds-\sigma_F
	=-\sum_{j=1}^n D_j(F)\,dx_j
	-\sum_{j=1}^m B g_j\,dz_j+ds
	\in H^0\bigl(M,q_M^*\Omega^1_{\mathbb A^N}\bigr).
\]
Then $\sigma_{M}\coloneqq \sigma_u/v$ is a rational section of $q_M^*\Omega_{\mathbb{A}^N}^1$, which is regular over the open subset 
\[
M_v\coloneqq W\times E_v,
\qquad E_v\coloneqq E\cap\{v\neq0\}.
\]

\begin{lemma}
	\label{lem.rational-1-form}
	We have $(dq_M)(\sigma_M)=\xi_M$. Moreover, for any normal projective birational model $M'$ of $M$, the $1$-form $\xi_M$ on $M$ extends uniquely to a regular $1$-form $\xi_{M'}$ on	$M'_{\mathrm{reg}}$.
\end{lemma}

\begin{proof}
	The equality $(dq_M)(\sigma_M)=\xi_M$ follows from the equality $u=s-F$. For the second statement, since $M'$ and $\overline{E}$ are projective, the rational map $\varphi\colon M'\dashrightarrow \overline{E}$ induced by $M\rightarrow E$ is regular in codimension one. In particular, there exists an open subset $U\subset M'$ such that $\varphi|_U$ is regular and $\operatorname{codim}(M'\setminus U)\geq 2$. Then $\xi_M=\pr^*_{E}(\xi|_E)$ extends uniquely to the regular $1$-form $\varphi^*\xi\in H^0(U,\Omega_U^1)$, and then $\varphi^*\xi$ extends to a regular $1$-form $\xi_{M'}\in H^0(M'_{\operatorname{reg}},\Omega^1_{M'_{\operatorname{reg}}})$ by Hartogs' theorem as $\operatorname{codim}(M'\setminus U)\geq 2$.
\end{proof}

We note that the rational section \(\sigma_M\) defines a rational map
\[
\begin{tikzcd}[row sep=large, column sep=large]
	f_M\colon M=W\times E \arrow[r, dashed, "{\sigma_M}"]
	    & q_M^*T^*\mathbb{A}^N \arrow[r]
	        &  T^*\mathbb A^N=\mathbb A^{2N}\ni (x,z,s,x',z',s'),
\end{tikzcd}
\]
given by
\[
	f_M(w,u,v)=
	\biggl(
	q(w),u+F(w),
	-\frac{D_1(F)}v,\ldots,-\frac{D_n(F)}v,
	-\frac{Bg_1}v,\ldots,-\frac{Bg_m}v,\frac1v
	\biggr).
\]

\begin{lemma}
	\label{lem.f'-embedding-codim-one}
	The morphism $f_{M_v}:=f_M|_{M_v}\colon M_v\to\mathbb A^{2N}$ is an embedding in codimension one.
\end{lemma}

\begin{proof}
	By Proposition \ref{prop:f-codim-one}, we can choose an open subset $W^\circ\subset W$ such that $\operatorname{codim}(W\setminus W^{\circ})\geq 2$ and $f|_{W^{\circ}}$ is an embedding. On $W^\circ\times E_v\subset M_v$, we have that $f_M(W^{\circ}\times E_v)$ is contained in the open subset $\{s'\not=0\}$ of $T^*\mathbb{A}^N\ni(x,z,s,x',z',s')$. In particular, we have
	\[
	v=\frac1{s'},\qquad
	f(w)=\left(x,z,-\frac{x'}{s'},-\frac{z'}{s'}\right),
	\qquad
	u=s-F(w).
	\]
	It follows that $f_M$ is an	embedding on $W^\circ\times E_v$, whose complement in $M_v=W\times E_v$ has codimension at least two.
\end{proof}

\begin{remark}
	Lemma \ref{lem.f'-embedding-codim-one} explains why we define $s=u+F$, instead of $s=u$. Indeed, the codimension one embedding property of the map $f_M$ (induced by $\xi_M$) would be lost in the latter case. Furthermore, this is why we need to choose $\sigma_F$, instead of an arbitrary section of $q^* T^*\mathbb{A}^{n+m}$, to make $f\colon W\rightarrow T^*\mathbb{A}^{n+m}$ into an \emph{exact Lagrangian embedding} in codimension one. In particular, the function $F$ can be replaced by any function satisfying the conclusion of Proposition \ref{prop:f-codim-one}.
\end{remark}

\subsection{From local to global}
\label{ss.Local-to-global}

Regard $\mathbb A^N$ as the standard affine open subset of $\mathbb P^N$. Let $\overline M$ be the normalization of \(\mathbb P^N\) in \(K(M)\), with induced finite morphism $q_{\overline M}\colon\overline M\longrightarrow\mathbb P^N$.
Since $M$ is normal and finite over \(\mathbb A^N\), it identifies
with \(q_{\overline M}^{-1}(\mathbb A^N)\). Let $\sigma_{\overline M}$ be the rational section of $q_{\overline M}^*\Omega^1_{\mathbb P^N}$ induced by $\sigma_M$. Then Lemma \ref{lem.rational-1-form} shows that $\xi_{\overline M}\coloneqq 
(dq_{\overline M})(\sigma_{\overline M})$ is regular on $\overline M_{\mathrm{reg}}$. Let $\mathfrak H$ be the pole divisor of $\sigma_{\overline M}$, and set
\[
\mathfrak D=
q_{\overline M}\bigl(\operatorname{Supp}\mathfrak H\bigr)
\cup(\mathbb P^N\setminus\mathbb A^N)
\]
with the reduced structure. Set $\mathfrak{D}_{\overline{M}}=q_{\overline{M}}^{-1}(\mathfrak{D})$ and write
\[
L\coloneqq \{0\}\times\mathbb A^m\times\mathbb A^1
\subset
\mathbb A^n\times\mathbb A^m\times\mathbb A^1
=\mathbb A^N.
\]
Then we have the following commutative diagram:
\[ 
\begin{tikzcd}[row sep=large, column sep=large]
	\{o_Y\}\times \mathbb{A}^m\times E \arrow[r,hook] \arrow[d,"{q_M}" left]
	& M \arrow[r,hook] \arrow[d,"{q_M}" left]
	&  \overline{M} \arrow[d,"{q_{\overline{M}}}"] 
	& \mathfrak{D}_{\overline{M}}\arrow[l,hook'] \supset \mathfrak{H}\arrow[d,"{q_{\overline{M}}}"]\\
	L \arrow[r,hook]
	& \mathbb{A}^N \arrow[r,hook]
	&  \mathbb{P}^N
	& \mathfrak{D} \arrow[l,hook'].
\end{tikzcd}
\]

\begin{lemma}
	\label{lem:Poles-Obs}
	The affine subspace $L$ of $\mathbb{A}^N$ is not contained in $\mathfrak D$.
\end{lemma}

\begin{proof}
	Over the open subset $M$, the pole locus of $\sigma_M=\sigma_u/v$ is contained in
	the closed subset $\{v=0\}$; that is,
	\[
	\operatorname{Supp}(\mathfrak H)\cap M \subset \{v=0\}.
	\]
	This implies
	\[
	\mathfrak D\cap L
	\subset
	q_M\bigl(\{p(y)=0,\ v=0\}\bigr).
	\]
	Consider the following projection 
	\[
	\pr_z\colon \mathfrak{D}\cap L \longrightarrow \mathbb{A}^m,\quad (0,z,s)\longmapsto z.
	\]
	We claim that $\pr_z$ is quasi-finite: Indeed, given $z$, since $p\colon Y\rightarrow \mathbb{A}^n$ is finite, there are only finitely many $y\in p^{-1}(0)$ and finitely many points of $E$ with $v=0$. Therefore $s=u+F(y,z)$ has only finitely many possible values. Hence $\dim(\mathfrak D\cap L)\leq m<\dim L$, proving the lemma.
\end{proof}

Choose a log resolution $\nu\colon\widetilde{\mathbb P}\to\mathbb P^N$ of $(\mathbb P^N,\mathfrak D)$, which is an isomorphism over $\mathbb P^N\setminus\mathfrak D$. Then the preimage $\mathfrak D_{\widetilde{\mathbb P}}
\coloneqq \nu^{-1}(\mathfrak D)$ is a simple normal crossing divisor. Let  $\widetilde M$ be the normalization of $\widetilde{\mathbb P}$ in $K(\overline M)$ with $\widetilde{\nu}\colon \widetilde{M}\rightarrow \overline{M}$.  Set $\mathfrak D_{\widetilde M}\coloneqq q_{\widetilde M}^{-1}(\mathfrak D_{\widetilde{\mathbb P}})$.
Pullback followed by the differential of \(\nu\) defines a rational section
\[
\sigma_{\widetilde M}
\coloneqq 
q_{\widetilde M}^*(d\nu)
\bigl(\widetilde\nu^*\sigma_{\overline M}\bigr)
\]
of $q_{\widetilde{M}}^*\Omega_{\widetilde{\mathbb{P}}}^1$ and its pole divisor is supported on
\(\mathfrak D_{\widetilde M}\). Moreover, the differential 
\[
\xi_{\widetilde{M}}\coloneqq (dq_{\widetilde M})(\sigma_{\widetilde M})
\]
is regular on $\widetilde M_{\mathrm{reg}}$ by Lemma \ref{lem.rational-1-form}. By Lemma \ref{lem:Poles-Obs}, choose a point $o\in L\setminus\mathfrak D$ and an open neighborhood $U\subset\mathbb P^N\setminus\mathfrak D\subset \mathbb{A}^N$
of $o$. Then we can identify $U$ with $\nu^{-1}(U)$. Set $U_M=q_M^{-1}(U)\subset M$. Then we can identify $U_M$ with an open subset of $\overline{M}$ and $\widetilde{M}$, respectively. Moreover, we can choose a point 
\[
\mathfrak{o}\in \{o_Y\}\times \mathbb{A}^m\times E \subset U_M
\]
such that $q_M(\mathfrak{o})=o$. In summary, we have the following commutative diagram:
\[
\begin{tikzcd}[row sep=large,column sep=large]
	\widetilde{M} \arrow[d,"{q_{\widetilde{M}}}" left] \arrow[rrrr,bend left,"{\widetilde{\nu}}"]
	   & \mathfrak{o}\in U_M\subset \widetilde M \setminus \mathfrak{D}_{\widetilde{M}}
	\arrow[r,"\widetilde\nu", "{\cong}"'] \arrow[l,hook']
	\arrow[d,"q_{\widetilde M}" left]
	& \mathfrak{o}\in U_M\subset \overline M \setminus \mathfrak{D}_{\overline{M}}
	\arrow[d,"q_{\overline M}"] \arrow[r,hook]
	& M \arrow[d,"{q_M}"] \arrow[r,hook]
	    &  \overline{M} \arrow[d,"{q_{\overline{M}}}"]\\
	\widetilde{\mathbb{P}} \arrow[rrrr,bend right, "{\nu}"]
	    & o\in U \subset \widetilde{\mathbb P} \setminus \mathfrak{D}_{\widetilde{\mathbb{P}}}
	\arrow[r,"\nu", "{\cong}"'] \arrow[l,hook']
	& o\in U \subset \mathbb P^N \setminus \mathfrak{D} \arrow[r,hook]
	    &  \mathbb{A}^N \arrow[r,hook]
	        & \mathbb{P}^N
\end{tikzcd}
\]

We now apply Kawamata's covering trick.

\begin{lemma}
	\label{lem:Kawamata-Cover}
	After shrinking $U$ around $o$, there exists a smooth finite
	cover $\tau\colon X\longrightarrow\widetilde{\mathbb P}$ such that:
	\begin{enumerate}
		\item the branch locus of $\tau$ is disjoint from $U$;
		
		\item if $\mathfrak{D}'_{\widetilde{\mathbb{P}}}$ is a component of $\mathfrak D_{\widetilde{\mathbb P}}$, $\mathfrak{D}'_X\subset X$ is	a prime divisor over $\mathfrak{D}'_{\widetilde{\mathbb{P}}}$, and $\mathfrak{D}'_{\widetilde M}\subset\widetilde M$ is a prime divisor over
		$\mathfrak{D}'_{\widetilde{\mathbb{P}}}$, then
		\[
		e\biggl(\mathfrak{D}'_{\widetilde M}\big/\mathfrak{D}'_{\widetilde{\mathbb{P}}}\biggr)
		\biggm| e\biggl( \mathfrak{D}'_X\big/\mathfrak{D}'_{\widetilde{\mathbb{P}}}\biggr).
		\]
	\end{enumerate}
\end{lemma}

\begin{proof}
	For each irreducible component $\mathfrak{D}^i_{\widetilde{\mathbb{P}}}$ of $\mathfrak{D}_{\widetilde{\mathbb{P}}}$, we define
	\[
	n_i\coloneqq 
	\operatorname{lcm}
	\biggl\{
	e\biggl(\mathfrak{D}^{ij}_{\widetilde M} \big/ \mathfrak{D}^i_{\widetilde{\mathbb{P}}}\biggr)\, \bigg|\,
	\mathfrak{D}^{ij}_{\widetilde M}\subset q_{\widetilde M}^{-1}\bigl(\mathfrak{D}^{i}_{\widetilde{\mathbb{P}}}\bigr)
	\text{ a prime divisor}
	\biggr\}.
	\]
	Then applying Kawamata's covering lemma \cite[Theorem 17]{Kawamata1981} (see also \cite[Proposition~4.1.12]{Lazarsfeld2004}) to $\bigl(\widetilde{\mathbb{P}},\mathfrak{D}_{\widetilde{\mathbb{P}}}\bigr)$ with multiplicities $(n_i)$ gives a smooth finite cover
	$\tau\colon X\to\widetilde{\mathbb P}$ whose ramification index
	along every divisor over $\mathfrak{D}_{\widetilde{\mathbb{P}}}^i$ is $n_i$. Moreover, we remark that the auxiliary branch
	divisors in the construction may be chosen not to contain $o$ so that, after shrinking $U$, the branch locus of $\tau$ is therefore disjoint from $U$.
\end{proof}

Let $S$ be the normalization of an irreducible component of the fiber product $X\times_{\widetilde{\mathbb P}}\widetilde M$. Denote by $\overline{\tau}\colon S\rightarrow \widetilde{M}$ and $\gamma\colon S\rightarrow X$ the natural projections. Set 
\[
U_X\coloneqq \tau^{-1}(U)\quad \text{and}\quad
U_S\coloneq \gamma^{-1}(U_X)=\overline{\tau}^{-1}(U_M).
\]
Then $U_X\rightarrow U$ and $U_S\rightarrow U_M$ are \'etale finite morphisms by Lemma \ref{lem:Kawamata-Cover}. Choose points $\mathfrak{o}_S\in U_S$ and $o_X\in U_X$ such that $\overline{\tau}(\mathfrak{o}_S)=\mathfrak{o}$ and $\tau(o_X)=o$. Then we have the following commutative diagram:
\[
\begin{tikzcd}[column sep=large,row sep=large]
	\mathfrak{o}_S\in U_S \arrow[r,hook] \arrow[d,"{\gamma}" left] \arrow[rrr,bend left]
	& S
	\arrow[r,"\overline\tau"]
	\arrow[d,"\gamma" left]
	& \widetilde M
	\arrow[d,"q_{\widetilde M}"] 
	    & U_M\ni \mathfrak{o} \arrow[l,hook'] \arrow[d,"{q_M}"]\\
	o_X\in U_X \arrow[r,hook] \arrow[rrr,bend right]
	    & X
	\arrow[r,"\tau"]
	& \widetilde{\mathbb P}
	    & U\ni o\arrow[l,hook'].
\end{tikzcd}
\]
Denote the degree of $\gamma$ by $r'$. Then $r'\leq 2r=\deg(q_{\widetilde{M}})$. Let $\sigma_S$ be the rational section of $\gamma^*\Omega_X^1$, defined by the pullback of $\sigma_{\widetilde{M}}$ followed by the differential of $\tau$:
\[
\sigma_S
\coloneqq 
\gamma^*(d\tau)
\bigl(\overline\tau^*\sigma_{\widetilde M}\bigr)
\]

\begin{lemma}
	\label{lem:sigma-S-regular}
	The section \(\sigma_S\) is regular.
\end{lemma}

\begin{proof}
	Note that $\sigma_{\widetilde{M}}$ is regular over $\widetilde{M}\setminus \mathfrak{D}_{\widetilde{M}}=\widetilde{M}\setminus q_{\widetilde{M}}^{-1}(\mathfrak{D}_{\widetilde{\mathbb{P}}})$. On the other hand, note that the differential $\xi_{\widetilde{M}}=d(q_{\widetilde{M}})(\sigma_{\widetilde{M}})$ is a regular $1$-form over $\widetilde{M}$, so it follows from Lemmas \ref{lem:Ramified-Index} and \ref{lem:Kawamata-Cover} that $\sigma_S$ is regular in codimension one. In particular, since $S$ is normal and $\gamma^*\Omega_X^1$ is locally free, by Hartogs' theorem, the section $\sigma_S$ is regular over $S$. 
\end{proof}

Let $f_S\colon S\to T^*X$ be the morphism defined by \(\sigma_S\); in other words, $f_S$ is the composition
\[
\begin{tikzcd}[row sep=large,column sep=large]
	f_S\colon S \arrow[r,"{\sigma_S}"]
	    & \gamma^*T^*X \arrow[r]
	        & T^*X.
\end{tikzcd}
\] 

\begin{theorem}\label{thm:global-counterexample}
	The morphism $f_S$ determines a generically regular spectral
	datum $s\in S_X^{r'}$ with $r'=\deg(\gamma)\leq 2r$
	such that $\mathcal H_X^{r'}(s)=\emptyset$.
\end{theorem}

\begin{proof}
	Since $\gamma$ is a finite morphism of degree $r'$, the push-forward $\mathscr E=\gamma_*\mathscr O_S$ is a reflexive coherent sheaf over $X$ of rank $r'$ \cite[Corollary 1.7]{Hartshorne1980}, and therefore $\mathscr{E}$ is locally free in codimension two since $X$ is smooth \cite[Corollary 1.4]{Hartshorne1980}. Over the locally free locus $X^{\circ}$ of $\mathscr{E}$, multiplication by \(\sigma_S\), followed by the projection
	formula, defines an integrable Higgs field
	\[
	\theta^{\circ}\colon\mathscr E|_{X^{\circ}}\longrightarrow
	\mathscr E|_{X^{\circ}}\otimes\Omega_{X^{\circ}}^1.
	\]
	By Hartogs' theorem, the characteristic coefficients of $\theta^{\circ}$ extend uniquely to $X$ and hence define a spectral datum $s\in S_X^{r'}$.
	
	Next we show that $s$ is generically regular. Over the open subset $U$, since $U_X\rightarrow U$ is \'etale, we have $T^*U_X\cong U_X\times_U T^*U$. On the other hand, note that $U_S$ is a connected component of $U_X\times_U U_M$ and the restriction $f_S|_{U_S}\colon U_S\longrightarrow T^*U_X$ is the restriction of the base change of
	\[
	f_{M_v}|_{U_M}\colon U_M\longrightarrow T^*U.
	\]
	In particular, since the latter is an embedding in codimension one by Lemma
	\ref{lem.f'-embedding-codim-one}, the same holds for
	\(f_S|_{U_S}\), and hence $s$ is generically regular.
	
	Finally we prove that $\mathcal{H}_X^{r'}(s)=\emptyset$. Assume to the contrary that $\mathcal H_X^{r'}(s)\neq \emptyset$. Then restricting to the open subset $U_X$ gives $\mathcal H_{U_X}^{r'}(s|_{U_X})\neq\emptyset$. Now applying Proposition \ref{prop:general-higgs-factors-through-Z} to $f_S|_{U_S}\colon U_S\rightarrow T^*U_X$ implies that $\mathcal{M}_{\cm}(U_S)\not=\emptyset$. Choose a rank one MCM sheaf $\mathscr{L}\in \mathcal{M}_{\cm}(U_S)$. Then, by the depth formula, the completion $\widehat{\mathscr{L}}_{\mathfrak{o}_S}$ of the stalk $\mathscr{L}_{\mathfrak{o}_S}$ of $\mathscr{L}$ at $\mathfrak{o}_S$ is a rank one MCM module over $\widehat{\mathscr{O}}_{U_S,\mathfrak{o}_S}$. On the other hand, since $U_S\rightarrow U_M$ is \'etale, there exists a natural isomorphism
	\[
	\widehat{\mathscr{O}}_{U_S,\mathfrak{o}_S}\cong \widehat{\mathscr{O}}_{U_M,\mathfrak{o}} = \widehat{\mathscr{O}}_{M,\mathfrak{o}}.
	\]
	This implies that $\widehat{\mathscr{O}}_{M,\mathfrak{o}}$ admits a rank one MCM module, which contradicts Lemma
	\ref{lem:Non-Exi-MCMmodules} and the choice $\mathfrak{o}\in \{o_Y\}\times \mathbb{A}^m\times E$. Therefore $\mathcal H_X^{r'}(s)=\emptyset$.
\end{proof}

\subsection{Explicit examples}

Let $(Y,o_Y)$ be one of the cones constructed in Examples \ref{ex:bidegree} and \ref{ex:ma}, with $o_Y$ being the vertex. 

\begin{lemma}\label{lem:completion-no-mcm}
	If $\dim Y\geq 4$, then $\widehat{\mathscr{O}}_{Y,o_Y}$ admits no rank one MCM modules.
\end{lemma}

\begin{proof}
	Since $o_Y$ is the unique singular point of $Y$, the local ring $\mathscr{O}_{Y,o_Y}$ satisfies $R_2$ and $S_3$ if $\dim Y\geq 4$. In particular, by \cite[Korollar 1.5]{Flenner1981}, the natural map $\Cl(\mathscr{O}_{Y,o_Y})\rightarrow \Cl(\widehat{\mathscr{O}}_{Y,o_Y})$ is surjective. The result then follows from Proposition \ref{prop:cone-acm} and the non-existence of rank one MCM sheaves on $Y$.
\end{proof}

\begin{example}[Degree and generators of the non-aCM example]\label{ex: gen nonacm}
	Let $n\geq 4$, $a\geq 1$ and $b\geq a+n+2$. Let $Y$ be the cone over the polarized manifold $(Z,\mathscr{O}_Z(1,1))$ given in Example \ref{ex:bidegree}, where $Z$ is a smooth hypersurface of bidegree $(a,b)$ in $\mathbb{P}^1\times \mathbb{P}^{n-1}$. Choose $p$ to be the homogeneous Noether normalization defined by
$n$ general sections
\[
x_1,\ldots,x_n\in H^0(\mathbb{P}^1\times \mathbb{P}^{n-1},\mathscr{O}(1,1)).
\]
Their common zero locus on $\mathbb{P}^1\times \mathbb{P}^{n-1}$ consists of $n$ reduced points 
$\Gamma=\{Q_1,\ldots,Q_n\}$.
We may assume that
$\Gamma\cap Z=\emptyset$ and that the points
$\operatorname{pr}_1(Q_i)\in\mathbb{P}^1$ are pairwise distinct.
The induced morphism $\overline{p}\colon Z\to\mathbb{P}^{n-1}$ is a finite cover, inducing a finite morphism $p\colon Y\rightarrow \mathbb{A}^n$ such that
\[
\deg(p) = \deg(\overline{p})
=(aH_1+bH_2)(H_1+H_2)^{n-1}
=a+(n-1)b.
\]
Since $(Z,\mathscr{O}_Z(1,1))$ is projectively normal, the ring $R\coloneqq A(Y)$ is generated (over $\mathbb{C}$) in degree one. Moreover,
\[
\dim_{\mathbb C}R_1
= h^0(Z,\mathscr{O}_Z(1,1)) = 
h^0\bigl(\mathbb P^1\times\mathbb P^{n-1},\mathscr O(1,1)\bigr)
=2n,
\]
whereas the degree-one part of
$S=\mathbb C[x_1,\ldots,x_n]$ has dimension $n$.
Hence a basis of $R_1/S_1$ gives a minimal set of algebra generators. Thus we may set $m=n$. \qed
\end{example}

\begin{example}[Degree and generators of Ma's example]\label{ex; degma}
Let $n\geq5$ be an integer and retain the notation of Example
\ref{ex:ma}: $N$ is an integer, $d_1,\ldots,d_{N-n+2}$ are
positive integers $\geq 2$, $Z\subset\mathbb P^N$ is the smooth complete
intersection of multidegree $(d_1,\ldots,d_{N-n+2})$, and $Y$ is the
affine cone over $Z\times\mathbb P^1$ with respect to
$\mathscr O_Z(1)\boxtimes\mathscr O_{\mathbb P^1}(1)$. Its graded
coordinate ring is the graded $\mathbb C$-algebra
\[
R\coloneqq A(Y)
\cong \bigoplus_{m=0}^{\infty}
H^0\bigl(Z\times\mathbb P^1,
\mathscr O_Z(m)\boxtimes\mathscr O_{\mathbb P^1}(m)\bigr),
\]
where $m$ ranges over the nonnegative integers. Choosing a general linear projection of the Segre embedding 
\[
Z\times \mathbb{P}^1\longrightarrow \mathbb{P}^N\times \mathbb{P}^1 \longrightarrow \mathbb{P}^{2N+1}
\]
yields a finite morphism $\overline{p}\colon Z\times \mathbb{P}^1\rightarrow \mathbb{P}^{n-1}$, which induces a finite morphism $p\colon Y\rightarrow \mathbb{A}^n$ such that 
\[
\deg(p)=\deg(\overline{p}) = \deg(\mathscr{O}_{Z}(1)\boxtimes\mathscr{O}_{\mathbb{P}^1}(1)) = (n-1)\prod_{i=1}^{N-n+2}d_i.
\]
Since $(Z\times \mathbb{P}^1,\mathscr{O}_Z(1)\boxtimes \mathscr{O}_{\mathbb{P}^1}(1))$ is projectively normal, the ring $R$ is generated in degree one, and
\[
R_1\cong 
H^0(Z,\mathscr O_Z(1))
\otimes
H^0(\mathbb P^1,\mathscr O_{\mathbb P^1}(1)).
\]
Therefore, we have that  $\dim_{\mathbb C}R_1=2(N+1)$ and $\dim_{\mathbb C}S_1=n$.
Thus we may choose $m=2N+2-n$. \qed
\end{example}

Now we are in the position to finish the proof of Theorem \ref{thm:Chen-Ngo-false}.

\begin{proof}[Proof of Theorem \ref{thm:Chen-Ngo-false}]
Let $(Y,o_Y)$ be either of the pointed cones in Examples
\ref{ex:bidegree} and \ref{ex:ma} with $n=\dim (Y)\geq 4$. Let
$p\colon Y\rightarrow\mathbb A^n$ be the homogeneous Noether
normalization chosen in the corresponding example, and let $r$ be its degree. Choose homogeneous elements $g_1,\ldots,g_m$ such that $g_1,\ldots,g_m$ form a minimal generating set of
$A(Y)\coloneqq \Gamma(Y,\mathscr O_Y)$ as an algebra over the polynomial ring
$A(\mathbb{A}^n)\coloneqq \Gamma(\mathbb A^n,\mathscr O_{\mathbb A^n})$.

By Lemma~\ref{lem:completion-no-mcm}, the completed local ring
$\widehat{\mathscr O}_{Y,o_Y}$ admits no rank one MCM module. Applying the construction from \S\,\ref{ss.Local-to-global} to $(Y,o_Y)$ yields a projective manifold $X$ of dimension $(n+m+1)$ such that by Theorem~\ref{thm:global-counterexample} there exists a generically regular spectral datum $s\in S_X^{r'}$ with $\mathcal{H}_X^{r'}(s)=\emptyset$, where $r'\leq 2r$, and we also have $r'\geq 3$ by \cite{HeLiu2024}.
\end{proof}

\begin{remark}
\label{rmk.Min-Possible}
    Let us record the minimal counterexamples using Examples \ref{ex:bidegree} and \ref{ex:ma}.
    \begin{enumerate}
        \item For the family of Example~\ref{ex: gen nonacm}, we have $r=a+(n-1)b$ and  $m\geq n$. Hence the smallest admissible choice in this family is $(n,a,b)=(4,1,7)$. It follows that for any $d\geq 9$, there exists a smooth projective $d$-fold $X$ and a $\operatorname{GL}_{r'}$-counterexample on $X$, where $r'\leq 44$. 

        \item For the family of Example~\ref{ex; degma}, the smallest case is a general sextic threefold $Z\subset\mathbb P^4$. In this case, we have that $r=24$, and that $m\geq 2\cdot4+2-5=5$.
Therefore, we obtain a $\operatorname{GL}_{r'}$-counterexample with $r'\leq 48$ on some smooth projective $d$-fold with $d\geq 11$.
    \end{enumerate}
\end{remark}

\section{$\operatorname{SL}_2$-examples}
\label{sec:sl2-counterexample}

The Chen--Ng\^o conjecture was formulated for $G$-Higgs bundles for reductive groups $G$, not only ordinary Higgs bundles (the $G=\operatorname{GL}_r$ case). For algebraic surfaces, it was studied in \cite{Huynh2025,SheshmaniWangXia2026} and proved for $G=\operatorname{SL}_{2r+1}$ in \cite{SheshmaniWangXia2026}. In this section, we construct counterexamples for $G=\operatorname{SL}_2$ in every dimension at least $2$.

\subsection{Setup}
\label{ss.Sl2-Setup}

We use the notation of Section~\ref{sec:hitchin-local}. Let $X$ be a projective manifold. The spectral base for $\operatorname{SL}_2$ is the
traceless locus
\[
S_X^{\operatorname{SL}_2}\coloneqq S_X^2\cap\bigl(\{0\}\times
H^0(X,\Sym^2\Omega_X^1)\bigr)\subset S_X^2.
\]
We denote by $\mathcal H_X^{\operatorname{SL}_2}$ the set of isomorphism
classes of rank two Higgs bundles $(\mathscr E,\theta)$ satisfying
$\det(\mathscr E)\cong\mathscr{O}_X$ and $\operatorname{tr}(\theta)=0$. The
corresponding spectral morphism is
$\sd_X^{\operatorname{SL}_2}\colon\mathcal H_X^{\operatorname{SL}_2}
\rightarrow S_X^{\operatorname{SL}_2}$, and for
$s\in S_X^{\operatorname{SL}_2}$ we write
$\mathcal H_X^{\operatorname{SL}_2}(s)$ for its fiber over $s$. Then the Chen--Ng\^o conjecture for $\operatorname{SL}_2$ predicts that
$\mathcal H_X^{\operatorname{SL}_2}(s)\neq\emptyset$ for every
$s\in S_X^{\operatorname{SL}_2}$.

Given a nonzero generically regular spectral datum
$s\in H^0(X,\Sym^2\Omega_X^1)$, there exists a line bundle $\mathscr L$, a
morphism $\alpha\colon\mathscr L\rightarrow\Omega_X^1$, and a section
$\tau\in H^0(X,\mathscr L^{\otimes2})$ such that
$\alpha(\mathscr L)$ is saturated in $\Omega_X^1$ and
\begin{equation}
\label{eq.decomp-s}
s=\alpha^2\tau,
\end{equation}
where $\alpha$ is viewed as a section of
$\Omega_X^1\otimes\mathscr L^{-1}$ \cite[Proposition~3.6]{HeLiu2024}.
The section $\tau$ determines a double cover
$p_s\colon\widetilde X_s\rightarrow X$ inside $\Tot(\mathscr L)$, and the
induced morphism
\[
\widetilde X_s\longrightarrow\Tot(\mathscr L)
\xlongrightarrow{\alpha}T^*X
\]
is a Cohen--Macaulayfication of the spectral variety $X_s$
\cite[Proposition~4.3]{HeLiu2024}. The modified spectral correspondence
identifies $\mathcal H_X^{\operatorname{SL}_2}(s)$ with the MCM sheaves $\mathscr M$ of generic rank one on $\widetilde X_s$ satisfying
$\det((p_s)_*\mathscr M)\cong\mathscr{O}_X$; compare
\cite[Theorem~4.7]{HeLiu2024}. If $\widetilde X_s$ is normal, this is also
the correspondence of Proposition~\ref{prop:general-higgs-factors-through-Z}
applied to $\widetilde X_s\rightarrow T^*X$, with the determinant condition
imposed.

\subsection{Endoscopic $\operatorname{SL}_2$-examples}

For a curve, a spectral datum is called \emph{endoscopic} if it lies in the
image of the square map
$H^0(X,\mathscr L\otimes\Omega_X^1)\rightarrow
H^0(X,\Sym^2\Omega_X^1)$ for some $2$-torsion line bundle $\mathscr L$;
see \cite[\S\,1]{maulik2021endoscopic}. 
The following criterion produces
counterexamples of this type.

\begin{proposition}
\label{prop.Critersion-SL2}
Let $X$ be a projective manifold of dimension $n$, and let
$p\colon Y\rightarrow X$ be a connected \'{e}tale double cover with
deck involution $\iota\colon Y\rightarrow Y$. Assume that
\begin{enumerate}
\item the N\'{e}ron--Severi group $\operatorname{NS}(Y)$ is torsion-free;
\item the action of $\iota$ on $\operatorname{NS}(Y)$ is trivial;
\item the action of $\iota$ on $H^0(Y,\Omega_Y^1)$ is multiplication by
$-1$; and
\item there exists a nonzero form $\beta\in H^0(Y,\Omega_Y^1)=H^0(Y,p^*\Omega_X^1)$ whose
vanishing locus has codimension at least two.
\end{enumerate}
Then there exists a nonzero generically regular spectral datum
$s\in S_X^{\operatorname{SL}_2}$ such that
$\mathcal H_X^{\operatorname{SL}_2}(s)=\emptyset$.
\end{proposition}

\begin{proof}
Let $\mathscr L$ be the nontrivial $2$-torsion line bundle on $X$ corresponding to $p$, so that
\[
p_*\mathscr{O}_Y\cong\mathscr{O}_X\oplus\mathscr L^{-1},\quad
\mathscr L\cong\mathscr L^{-1}\quad \text{and}\quad p^*\mathscr L\cong\mathscr{O}_Y.
\]
Since $p$ is \'{e}tale, the projection formula gives
$p_*\Omega_Y^1\cong\Omega_X^1\oplus
(\Omega_X^1\otimes\mathscr L^{-1})$. The hypothesis on the involution
therefore yields
\[
H^0(X,\Omega_X^1)=0\quad \text{and}\quad
H^0(Y,\Omega_Y^1)\cong
H^0(X,\Omega_X^1\otimes\mathscr L^{-1}).
\]

The form $\beta$ descends to a nonzero section
$\alpha\in H^0(X,\Omega_X^1\otimes\mathscr L^{-1})$. It induces a
morphism $\alpha\colon\Tot(\mathscr L)\rightarrow T^*X$ which is an
embedding away from the vanishing locus of $\beta$. The natural
embedding $Y\hookrightarrow\Tot(\mathscr L)$ is stable under fiberwise
multiplication by $-1$, and the composite
$f\colon Y\rightarrow\Tot(\mathscr L)\rightarrow T^*X$ satisfies
$f^*\lambda_{\mathrm{taut}}=\beta$. Proposition~\ref{prop:general-higgs-factors-through-Z}
gives a nonzero generically regular datum in $S_X^2$. Since the two sheets
of $Y\rightarrow X$ are exchanged by fiberwise multiplication by $-1$,
its first characteristic coefficient is zero. We denote the resulting
element of $S_X^{\operatorname{SL}_2}$ by $s$.

Suppose that $(\mathscr E,\theta)\in
\mathcal H_X^{\operatorname{SL}_2}(s)$. By
Proposition~\ref{prop:general-higgs-factors-through-Z}, there is a rank one MCM sheaf $\mathscr M$ on $Y$ such that
$\mathscr E\cong p_*\mathscr M$. Since $Y$ is smooth, $\mathscr M$ is a
line bundle. Moreover, we have
\[
\mathscr{O}_X\cong\det(p_*\mathscr M)
\cong\det(p_*\mathscr{O}_Y)\otimes\operatorname{Nm}_{Y/X}(\mathscr M)
\cong\mathscr L^{-1}\otimes\operatorname{Nm}_{Y/X}(\mathscr M),
\]
and hence $\operatorname{Nm}_{Y/X}(\mathscr M)\cong\mathscr L$. Pulling
back to $Y$, we obtain
$\mathscr M\otimes\iota^*\mathscr M\cong
p^*\operatorname{Nm}_{Y/X}(\mathscr M)\cong\mathscr{O}_Y$. Thus
$c_1(\mathscr M)+\iota^*c_1(\mathscr M)=0$ in
$\operatorname{NS}(Y)$. Since $\iota$ acts trivially, this says
$2c_1(\mathscr M)=0$, and the torsion-freeness of $\operatorname{NS}(Y)$ now gives $c_1(\mathscr M)=0$.
Consequently, $\mathscr M\in\Pic^0(Y)$. However, the norm map sends $\Pic^0(Y)$ to
$\Pic^0(X)$, while $H^0(X,\Omega_X^1)=0$ implies
$\Pic^0(X)=0$. Hence
$\operatorname{Nm}_{Y/X}(\mathscr M)\cong\mathscr{O}_X$, contradicting the facts that
$\operatorname{Nm}_{Y/X}(\mathscr M)\cong\mathscr L$ and
$\mathscr L$ is not isomorphic to $\mathscr{O}_X$.
\end{proof}

\subsubsection{Canonically polarized projective manifolds}

Now we construct explicit examples of canonically polarized projective manifolds, i.e. a projective manifold $X$ with $K_X$ ample, satisfying the hypothesis in Proposition \ref{prop.Critersion-SL2} in every dimension $\geq 2$.

\begin{theorem}
\label{thm:sl2-counterexample}
For every $n\geq2$, there exists a projective manifold $X$ of
dimension $n$ with $K_X$ ample and a generically regular spectral datum
$s\in S_X^{\operatorname{SL}_2}$ such that
$\mathcal H_X^{\operatorname{SL}_2}(s)=\emptyset$. In particular, the
Chen--Ng\^o conjecture for $\operatorname{SL}_2$ fails on such $X$.
\end{theorem}

\begin{proof}
Let $A$ be a very general principally polarized abelian $(n+1)$-fold, and
fix $a\in A[2]$. Define an involution $\iota_A$ on $A$ by
$\iota_A(x)=-x+a$. Its fixed locus is finite. Choose a sufficiently
positive $\iota_A$-linearized very ample line bundle $H$ such that the
invariant linear system $|H|^{\iota_A}$ is basepoint-free. Let $Y$ be a
very general member of $|H|^{\iota_A}$. By Bertini's theorem, $Y$ is
smooth, and we may choose it disjoint from $\operatorname{Fix}(\iota_A)$.
Thus $\iota\coloneqq\iota_A|_Y$ acts freely. Set
$X\coloneqq Y/\langle\iota\rangle$, and let
$p\colon Y\rightarrow X$ be the quotient map. By adjunction, $K_Y\cong H|_Y$ is ample. Since $p$ is finite and
\'{e}tale, $p^*K_X\cong K_Y$, and therefore $K_X$ is ample.

The classical Lefschetz theorem shows that
$H^2(A,\mathbb Z)\rightarrow H^2(Y,\mathbb Z)$ is an isomorphism if
$n\geq3$ and an injection with torsion-free cokernel if $n=2$
\cite[Theorem~3.1.17 and Example~3.1.18]{Lazarsfeld2004}. Hence
$H^2(Y,\mathbb Z)$, and therefore $\operatorname{NS}(Y)$, is
torsion-free. 

For $n\geq3$, the weak Lefschetz theorem gives
$\operatorname{NS}(A)_{\mathbb Q}\cong\operatorname{NS}(Y)_{\mathbb Q}$.
For $n=2$, the same conclusion follows from an equivariant version of the Noether--Lefschetz theorem established in
Corollary~\ref{cor: finite group surface noether lefschetz}.
Since $A$ is very general, $\operatorname{NS}(A)\cong\mathbb Z$.
Moreover, $\iota_A$ acts trivially on $H^2(A,\mathbb Z)$. It follows that
$\iota$ acts trivially on $\operatorname{NS}(Y)\cong \mathbb{Z}$. 

The natural restriction induces an isomorphism
$H^0(Y,\Omega_Y^1)\cong H^0(A,\Omega_A^1)$. Since $\iota_A$ acts as
$-1$ on the translation-invariant one forms on $A$, the action of
$\iota$ on $H^0(Y,\Omega_Y^1)$ is multiplication by $-1$.

Finally, let $\beta$ be a general element of
$H^0(Y,\Omega_Y^1)\cong H^0(A,\Omega_A^1)$. We claim that its vanishing
locus in $Y$ has codimension at least two. Let $N_{Y/A}^*$ be the
conormal bundle of $Y$ in $A$. The form $\beta$ vanishes at $y\in Y$ if
and only if the line $N_{Y/A,y}^*$ is spanned by $\beta(y)$. Consider the
composite
\[
Y\cong\mathbb P(N_{Y/A}^*)\longrightarrow
\mathbb P(\Omega_A^1|_Y)\cong Y\times\mathbb P^n
\longrightarrow\mathbb P^n.
\]
The pullback of $\mathscr{O}_{\mathbb P^n}(1)$ is
$N_{Y/A}\cong H|_Y$, which is ample, so this composite is finite. The
form $\beta$ determines a point of $\mathbb P^n$, and its vanishing locus
on $Y$ is the fiber over that point. It is therefore finite and has
codimension at least two. Proposition~\ref{prop.Critersion-SL2} now
applies to $p\colon Y\rightarrow X$.
\end{proof}

\subsubsection{$K$-trivial projective manifolds}
\label{sec: k triv}

Chen--Ng\^o's conjecture for $\operatorname{GL}_r$-Higgs bundles over $K$-trivial projective manifolds, i.e., a projective manifold $X$ with $K_X$ numerically trivial, is confirmed in \cite{PatelWeissmann2026b}. In the following, we construct counterexamples on $K$-trivial projective manifolds for $\operatorname{SL}_2$ in every dimension $\geq 3$ using Proposition \ref{prop.Critersion-SL2}.

\begin{theorem}
\label{thm: k trivial counterexamples}
For every $n\geq3$, there exists a smooth projective $n$-fold $X$ with
$\omega_X\cong\mathscr{O}_X$ and a generically regular spectral datum
$s\in S_X^{\operatorname{SL}_2}$ such that
$\mathcal H_X^{\operatorname{SL}_2}(s)=\emptyset$.
\end{theorem}

\begin{proof}
Let $S$ be a general Enriques surface, and let $K\rightarrow S$ be its
$K3$-cover with covering involution $\sigma$. By
\cite[Propositions~2.3 and~5.6]{Namikawa1985}, every line bundle on $K$
is $\sigma$-invariant. Let $E$ be an elliptic curve and set
$Y_0\coloneqq K\times E$ and
$\iota_0\coloneqq\sigma\times[-1]_E$. Since $\sigma$ acts freely, so does
$\iota_0$. Thus $X_0\coloneqq Y_0/\langle\iota_0\rangle$ is a smooth
projective threefold, and the quotient map
$p_0\colon Y_0\rightarrow X_0$ is an \'{e}tale double cover. Choose nonzero forms $\zeta_K\in H^0(K,\omega_K)$ and
$\eta\in H^0(E,\Omega_E^1)$. The involutions $\sigma$ and $[-1]_E$ act
by $-1$ on $\zeta_K$ and $\eta$, respectively. Hence
$\zeta_K\wedge\eta$ is an invariant nowhere-vanishing volume form on
$Y_0$ and descends to $X_0$, proving $\omega_{X_0}\cong\mathscr{O}_{X_0}$.

By the product formula for Picard groups
\cite[Chapter~III, Exercise~12.6]{Hartshorne1977}, the group
$\operatorname{NS}(Y_0)$ is torsion-free, and $\iota_0$ acts trivially on
it: $\sigma$ acts trivially on $\Pic(K)$, while $[-1]_E$
acts trivially on $\operatorname{NS}(E)$. Moreover,
$H^0(Y_0,\Omega_{Y_0}^1)\cong H^0(E,\Omega_E^1)$, and $\iota_0$ acts on
this space by $-1$. The nowhere-vanishing one form
$\pr_E^*\eta$ therefore satisfies
Proposition~\ref{prop.Critersion-SL2}. We obtain a generically regular
spectral datum $s_0\in S_{X_0}^{\operatorname{SL}_2}$ with
$\mathcal H_{X_0}^{\operatorname{SL}_2}(s_0)=\emptyset$.

Let $B$ be an abelian variety of dimension $n-3$, with $B$ a point when
$n=3$, and put $X\coloneqq X_0\times B$. Then $\omega_X\cong\mathscr{O}_X$, and the pullback
$s\coloneqq\pr_{X_0}^*s_0$ is generically regular. If it were realized by
an $\operatorname{SL}_2$-Higgs bundle, restriction to
$X_0\times\{b\}$, followed by the projection
$\Omega_X^1|_{X_0\times\{b\}}\rightarrow\Omega_{X_0}^1$, would produce
an $\operatorname{SL}_2$-Higgs bundle with spectral datum $s_0$, a
contradiction.
\end{proof}

\subsection{Non-endoscopic $\operatorname{SL}_2$-examples}

We next construct counterexamples where
the factor $\tau$ in \eqref{eq.decomp-s} is not a square.

\begin{theorem}\label{thm: non-endo}
    For every $n\geq 3$, there exists a smooth projective $n$-fold $X$ and a generically regular spectral datum $s\in S_X^{\operatorname{SL}_2}$ such that
    \begin{enumerate}
        \item there does not exist a $2$-torsion line bundle $\mathscr{T}$ on $X$ such that $s$ lies in the image of the square morphism $H^0(X,\mathscr{T}\otimes \Omega_X^1)\to H^0(X,\operatorname{Sym}^2\Omega_X^1)$, and
        \item the Hitchin fiber $\mathcal H_X^{\mathrm{SL}_2}(s)=\emptyset$.
    \end{enumerate}
\end{theorem}
\begin{proof}
We first construct a $3$-dimensional example. Let $C$ be a smooth hyperelliptic curve of genus two, let $h\colon C\rightarrow\mathbb P^1$
be the hyperelliptic double cover, and let $\sigma$ be its hyperelliptic
involution. We have that $\omega_C\cong h^*\mathscr{O}_{\mathbb P^1}(1)$.
Choose a section $s_0\in H^0\bigl(\mathbb P^1,\mathscr{O}_{\mathbb P^1}(2)\bigr)$
whose zero divisor consists of two distinct points disjoint from the branch
locus of $h$, and set $q\coloneq h^*s_0\in H^0(C,\omega_C^{\otimes 2})$. Let $Z$ be the $K3$-cover of an Enriques surface, and let
$\iota\colon Z\to Z$ be the corresponding involution.
Then the diagonal involution $g\coloneqq \sigma\times\iota$
acts freely on $C\times Z$. Set $X\coloneqq (C\times Z)/\langle g\rangle$,
and denote the quotient map by $\rho\colon C\times Z\to X$.
\bigskip

\noindent\textit{Step 1: Construction of the spectral datum $s$.}
\smallskip

The naturally $g$-linearized line bundle
$\operatorname{pr}_C^*\omega_C$ descends to a line bundle $\mathscr{L}$ on $X$.
The natural inclusion $\operatorname{pr}_C^*\omega_C
\hookrightarrow
\Omega_{C\times Z}^1$
is $g$-equivariant. 
It descends to a line subbundle $\alpha\colon\mathscr{L}\hookrightarrow\Omega_{X}^1$.
The section $\operatorname{pr}_C^*q$ descends to a section $\tau\in H^0(X,\mathscr{L}^{\otimes 2})$, so we get a generically regular spectral datum $s\in S_X^{\operatorname{SL}_2}$, defined explicitly as follows:
\[
s\coloneqq \alpha^2\tau\in
H^0(X,\operatorname{Sym}^2\Omega_{X}^1).
\]
\smallskip

\noindent\textit{Step 2: The spectral datum $s$ is not endoscopic.}
\smallskip

Suppose that $s=\beta^2$ for some $2$-torsion line bundle $\mathscr{T}$ and some $\beta\in H^0(X,\Omega_{X}^1\otimes \mathscr{T})$. Let $\mathscr{L}_\beta\subset\Omega_{X}^1$ be the saturated rank one
subsheaf generated by $\beta(\mathscr{T}^{-1})$. Then $\beta$ factors through a section 
\[
\widetilde\beta\in H^0(X,\mathscr{L}_\beta\otimes \mathscr{T}).
\]
Since $\beta^2=s=\alpha^2\tau$, the tensors $\beta$ and $\alpha$ determine
the same line in $\Omega_{X}^1$ at the generic point of $X$. 
As both $\mathscr L_\beta$ and $\alpha(\mathscr L)$ are saturated, they coincide.
Therefore, after identifying $\mathscr L_\beta$ with $\mathscr L$, we have $\tau=\widetilde\beta^{\,2}$, so $\operatorname{div}(\tau)$ is non-reduced. However, since $\rho^*\tau=\operatorname{pr}_C^*q$ and $\operatorname{div}(q)$ is a reduced divisor, we get a contradiction.
\bigskip

\noindent\textit{Step 3: The Hitchin fiber over $s$ is empty.}
\smallskip

Let $p\colon \widetilde{X}_s\to X$ be the Cohen--Macaulayfication of the spectral variety corresponding to $s$ introduced in the beginning of this section. Since $\operatorname{div}(\tau)$ is smooth, the variety $\widetilde{X}_s$ is a smooth double cover (corresponding to $\tau$) of $X$. Let $\pi\colon C_q\to C$ be the spectral curve defined by $q$. Here \(\widetilde X_s\subset\operatorname{Tot}(\mathscr L)\) and \(C_q\subset\operatorname{Tot}(\omega_C)\) are defined by \(\eta^2+\tau=0\) and \(\xi^2+q=0\), respectively. Since \(q\) is invariant under the natural action of \(\sigma\) on \(\omega_C^{\otimes2}\), the canonical action on \(\omega_C\) induces an involution \(\widehat\sigma\) on \(C_q\), satisfying \(\pi\circ\widehat\sigma=\sigma\circ\pi\).
Since $\operatorname{div}(q)$ is reduced, the curve $C_q$ is smooth. Moreover, as $\rho^*\mathscr{L}\cong \pr_C^* \omega_C$ and $\rho^*\tau=\pr_C^* q$, by the functoriality of the construction of double covers, we have the following Cartesian diagram:
\[
\begin{tikzcd}[column sep=large,row sep=large]
C_q\times Z
    \arrow[r,"r"]
    \arrow[d,"{\pi\times\operatorname{id}_Z}"']
&
\widetilde X_s
    \arrow[d,"p"]
\\
C\times Z
    \arrow[r,"\rho"']
&
X.
\end{tikzcd}
\]
The map \(r\) is the quotient by the involution \(\widehat\sigma\times\iota\).

Assume to the contrary that $s$ were realized by an $\mathrm{SL}_2$-Higgs bundle.
Then, there would exist
$\mathscr{M}\in\operatorname{Pic}(\widetilde{X}_s)$ such that $\det(p_*\mathscr{M})\cong\mathscr O_{X}$ by Proposition \ref{prop:general-higgs-factors-through-Z} (see also \cite[Theorem 4.6]{HeLiu2024}).
As in the proof of Proposition \ref{prop.Critersion-SL2}, the determinant formula gives $\operatorname{Nm}_{\widetilde{X}_s/X}(\mathscr{M})\cong\mathscr L$. Since $\rho\colon C\times Z\to X$ is a $\mathbb{Z}/2\mathbb{Z}$-torsor, we have a $\mathbb{Z}/2\mathbb{Z}$-equivariant isomorphism 
\[
\rho^*\operatorname{Nm}_{\widetilde{X}_s/X}(\mathscr{M})\cong \rho^*\mathscr L.
\]
\bigskip

\noindent\textbf{Claim.} \textit{We have the following isomorphism which is equivariant with respect to the $\mathbb{Z}/2\mathbb{Z}$-action defined by $\rho$:}
\[
\operatorname{Nm}_{C_q\times Z/C\times Z}(r^*\mathscr{M})
\cong
\operatorname{pr}_C^*\omega_C.
\]
\smallskip

\noindent\textit{Proof of the claim.} It suffices to show that the natural isomorphisms 
\[
\operatorname{Nm}_{C_q\times Z/C\times Z}(r^*\mathscr{M})\cong \rho^*\operatorname{Nm}_{\widetilde{X}_s/X}(\mathscr{M})\quad \text{and}\quad
\rho^*\mathscr{L}\cong \operatorname{pr}_C^*\omega_C
\]
are $\mathbb{Z}/2\mathbb{Z}$-equivariant. The first equivariance follows from the fact that for any coherent sheaf $\mathscr{F}$ on $\widetilde{X}_s$,  the canonical flat base-change isomorphism
$(\pi\times\operatorname{id}_Z)_*(r^*\mathscr F)
\cong
\rho^*(p_*\mathscr F)$ is $\mathbb{Z}/2\mathbb{Z}$-equivariant, and the second equivariance follows by construction: $\mathscr L$ is defined as
the descent along $\rho$ of $\operatorname{pr}_C^*\omega_C$. \qed
\bigskip

Since $Z$ is a $K3$ surface, we can apply the product formula for Picard groups \cite[Chapter~III, Exercise~12.6]{Hartshorne1977} and get line bundles $\mathscr{N}\in\operatorname{Pic}(C_q)$
and $\mathscr{P}\in\operatorname{Pic}(Z)$ such that $r^*\mathscr{M} \cong\mathscr{N} \boxtimes \mathscr{P}$.
Taking norms and restricting to a fiber \(\{c\}\times Z\) of \(\operatorname{pr}_C\colon C\times Z\to C\) yields
$\mathscr{P}^2\cong\mathscr O_Z$. Since $Z$ is a $K3$ surface, we
obtain $\mathscr{P}\cong\mathscr O_Z$. Therefore, we have that $r^*\mathscr{M}\cong \operatorname{pr}_{C_q}^*\mathscr{N}$.

Now, transport the canonical $\widehat\sigma\times\iota$-linearization of $r^*\mathscr M$ through the preceding isomorphism. This gives an isomorphism $\operatorname{pr}_{C_q}^*(\widehat\sigma^*\mathscr N)
\cong
\operatorname{pr}_{C_q}^*\mathscr N$.
Because $H^0(Z,\mathscr O_Z)=\mathbb C$,
this is the pullback of a unique isomorphism
\[
\begin{tikzcd}[column sep=large]
    \widehat{\sigma}^*\mathscr{N} \arrow[r,"{\cong}"]
       &  \mathscr{N}
\end{tikzcd}
\]
The cocycle condition upstairs implies the cocycle condition downstairs. Therefore $\mathscr N$ is naturally $\widehat\sigma$-linearized.
Combined with $\mathscr P\cong\mathscr O_Z$, the claim above becomes $\operatorname{pr}_C^*
\operatorname{Nm}_{C_q/C}(\mathscr N)
\cong
\operatorname{pr}_C^*\omega_C$ equivariantly. Again using $H^0(Z,\mathscr O_Z)=\mathbb C$, this is pulled back from a $\sigma$-equivariant isomorphism $\operatorname{Nm}_{C_q/C}(\mathscr N)\cong \omega_C$.

However, this is impossible. Choose a Weierstrass point $w\in C$.
Since $q(w)\neq0$, the fiber $\pi^{-1}(w)=\{x_+,x_-\}$
consists of two distinct points, which are exchanged by
$\widehat\sigma$. Hence the induced action on $(\pi_*\mathscr N)_w=
\mathscr N_{x_+}\oplus\mathscr N_{x_-}$
has determinant $-1$. The same argument shows that the action on
$\det(\pi_*\mathscr O_{C_q})_w$ has weight $-1$. Therefore the
action on 
\[
\operatorname{Nm}_{C_q/C}(\mathscr N)_w
=
\det(\pi_*\mathscr N)_w
\otimes
\det(\pi_*\mathscr O_{C_q})_w^{-1}
\]
has weight $+1$. On the other hand, the canonical action of $\sigma$ on $\omega_{C,w}$ has weight $-1$, since $d\sigma_w=-1$, which is a contradiction.

\bigskip

\noindent\textit{Step 4: For $n>3$, we can take products of $X$ with $\mathbb{P}^{n-3}$.}

\end{proof}

\section{Polystable example}
\label{sec:polystable}

Let $(X,H)$ be a polarized projective manifold. We write $\mathcal H_{X,H}^{r,\mathrm{ss}}$ and $\mathcal H_{X,H}^{r,\mathrm{ps}}$ for the loci of $\mu_H$-semistable and $\mu_H$-polystable rank $r$ Higgs bundles, respectively. 
The spectral morphism therefore restricts to a map $\sd_{X,H}^{r,\mathrm{ps}}\colon \mathcal H_{X,H}^{r,\mathrm{ps}}\rightarrow S_X^r$. 
It is shown in \cite[Corollary~1.3]{HeLiu2024} that this map is surjective in rank two. 
In this section, we show that the analogous statement fails in rank three for surfaces.

The construction is geometric in nature: We choose an algebraic surface $X$ and a spectral datum $s$ on $X$ so that its associated spectral variety $X_s$ separates into components of degrees one and two over $X$. We arrange the degree two cover so that the slope of the rank two summand can never equal the slope of the rank one summand. 

\subsection{A parity obstruction}

We begin with a numerical criterion for producing the spectral datum in Theorem~\ref{thm:polystable-counterexample}. Let $X$ be a smooth projective surface satisfying the following conditions:
\begin{itemize}
    \item The cotangent bundle admits a decomposition $\Omega_X^1\simeq\mathscr A\oplus\mathscr{O}_X\vartheta$, where $\mathscr A$ is a line bundle and $\vartheta$ is a nowhere vanishing $1$-form spanning the trivial summand.

    \item There exists $a\in H^0(X,\mathscr A^{\otimes 2})$ such that the associated double cover $\pi\colon Y\rightarrow X$ is smooth and connected.
\end{itemize}
Let $\eta\in H^0(Y,\pi^*\mathscr A)$ be the tautological section, so that $\eta^2=\pi^*a$. It defines a closed immersion $Y\hookrightarrow\Tot(\mathscr A)\hookrightarrow T^*X$. Its image is disjoint from the graph $\Gamma_\vartheta\subset T^*X$. The finite cover $\Gamma_\vartheta\sqcup Y\rightarrow X$ therefore determines a generically regular spectral datum $s\in S_X^3$.

\begin{proposition}\label{prop:parity-obstruction}
In the situation above, assume that $c_1(\mathscr A)\cdot H$ is odd and that $c_1(\mathscr L)\cdot\pi^*H$ is even for every line bundle $\mathscr L$ on $Y$. Then $\mathcal H_{X,H}^{3,\mathrm{ss}}(s)=\emptyset$.
\end{proposition}

\begin{proof}
Set $Z\coloneqq \Gamma_\vartheta\sqcup Y$. Since $Z$ is smooth and its natural map to $T^*X$ is a closed immersion, the spectral correspondence of Proposition~\ref{prop:general-higgs-factors-through-Z} entails that every Higgs bundle in the fiber of $s$ is of the form
$(\mathscr L_1,\vartheta)\oplus(\pi_*\mathscr L_2,m_\eta)$
for some $\mathscr L_1\in\Pic(X)$ and $\mathscr L_2\in\Pic(Y)$, where $m_\eta$ denotes multiplication by $\eta$.

Suppose that one of these Higgs bundles is $\mu_H$-semistable. Its two summands are Higgs-invariant, so they must have the same slope. The rank one summand has integral slope, since $\mu_H(\mathscr L_1)=c_1(\mathscr L_1)\cdot H$. On the other hand, we have $\det(\pi_*\mathscr L_2)\simeq\operatorname{Nm}_\pi(\mathscr L_2)\otimes\mathscr A^{-1}$, and hence
\[
\deg_H(\pi_*\mathscr L_2)
=
c_1(\mathscr L_2)\cdot\pi^*H-c_1(\mathscr A)\cdot H.
\]
The first term is even by assumption and the second is odd. Thus $\deg_H(\pi_*\mathscr L_2)$ is odd, so $\mu_H(\pi_*\mathscr L_2)$ is a half-integer. It cannot equal the integral slope of $\mathscr L_1$, which contradicts semistability.
\end{proof}

\subsection{An explicit example of $(X,H)$}

By Proposition \ref{prop:parity-obstruction}, we need to construct a polarized algebraic surface $(X,H)$ satisfying the conditions from the previous section. In the following we will construct $X$ as an \'etale quotient of a surface $E\times C$, where $E$ is an elliptic curve and $C$ is a projective curve with genus two. 

\subsubsection{Construction of the surface $X$}

Let $G\cong (\mathbb Z/2\mathbb Z)^2$, and write its nonzero elements as $\epsilon_1,\epsilon_2,\epsilon_3$, with $\epsilon_1+\epsilon_2+\epsilon_3=0$. Let $B\cong \mathbb P^1$, and choose a connected $G$-Galois cover $q\colon C\rightarrow B$ branched at five points $b_1,\ldots,b_5$, with local monodromies $\epsilon_1,\epsilon_1,\epsilon_1,\epsilon_2,\epsilon_3$. Such a cover exists by the Riemann existence theorem, and the Riemann--Hurwitz formula gives $g(C)=2$. The involution $\epsilon_1$ has six fixed points on the genus two curve $C$, since it occurs as the local monodromy over three branch points of $q\colon C\to B$. It is therefore the hyperelliptic involution of $C$. If $h\colon C\rightarrow C/\langle\epsilon_1\rangle\cong \mathbb P^1$ denotes the hyperelliptic map, then $q$ factors as the composition of $h$ with a degree two map $\mathbb P^1\rightarrow B$. Since $\omega_C\cong  h^*\mathscr{O}_{\mathbb P^1}(1)$, it follows that $q^*\mathscr{O}_B(1)\cong \omega_C^{\otimes 2}$.

Choose a point $b_0\in B\setminus\{b_1,\ldots,b_5\}$, and let
$t_0\in H^0(B,\mathscr{O}_B(1))$ be a non-zero section vanishing at $b_0$. We can regard $a_C\coloneqq q^*t_0$ as a section of $\omega_C^{\otimes2}$, and let
$r\colon C'\rightarrow C$ be the double cover defined by
$(\omega_C,a_C)$. The natural $G$-linearization of $\omega_C$ lifts the action of $G$ to $C'$.
This action is free. Indeed, if a nonzero element of $G$ fixes a point of
$C$, then it acts by $-1$ on the corresponding fiber of $\omega_C$ and
therefore exchanges the two points of $C'$ above it. We set $B'\coloneqq C'/G$.
The deck involution of $r$ descends to a double cover
$\beta\colon B'\rightarrow B$, branched at $b_0,b_1,\ldots,b_5$, and
$C'\rightarrow B'$ is an étale $G$-cover. In particular, we have the following commutative diagram:
\[
\begin{tikzcd}[row sep=large, column sep=large]
        & C' \arrow[dl,"{r}" above]    \arrow[dr]
            &    \\
    C \arrow[dr,"q"] \arrow[d,"h"]   
        & 
            & B'\cong C'/G  \arrow[dl,"{\beta}"]\\
    \mathbb{P}^1=C/\langle \epsilon_1\rangle \arrow[r]
        & C/G=B\cong \mathbb{P}^1.
            &
\end{tikzcd}
\]

Choose an elliptic curve $E$ such that $\operatorname{Hom}(J(C'),E)=0$, and fix an identification $G\cong E[2]$. We let $G$ act on $E$ by translations and diagonally on $E\times C$ and $E\times C'$. These diagonal actions are free. We then obtain two smooth projective surfaces:
\[
X\coloneqq (E\times C)/G
\quad \text{and}\quad Y\coloneqq (E\times C')/G.
\]
The natural maps induce genus one fibrations $f\colon X\rightarrow B$ and $g\colon Y\rightarrow B'$, together with a finite double cover $\pi\colon Y\rightarrow X$ lying over $\beta$:
\[
\begin{tikzcd}[column sep=large,row sep=large]
Y=(E\times C')/G \arrow[r,"\pi"] \arrow[d,"g"']
  & X=(E\times C)/G \arrow[d,"f"] \\
B' \arrow[r,"\beta"']
  & B .
\end{tikzcd}
\]

\begin{lemma}\label{lem:geometry-of-quotient-surface}
 The fibration $f$ has precisely five multiple fibers $2R_1,\ldots,2R_5$, where each $R_i$ is a smooth elliptic curve, while $g$ is a smooth isotrivial genus one fibration. 
 
 Moreover, let $F$ denote the class of a general fiber of $f$. Then, there is a nowhere vanishing $1$-form $\vartheta$ on $X$ such that
\[
\Omega_X^1\cong\mathscr{O}_X\vartheta\oplus\omega_X,
\qquad
\omega_X^{\otimes 2}\cong \mathscr{O}_X(F).
\]
\end{lemma}

\begin{proof}
The quotient map $E\times C\rightarrow X$ is étale. Over a branch point $b_i$, the inertia group of $C\rightarrow B$ has order two and acts on $E$ by translation. It follows that the fiber of $f$ over $b_i$ is twice a smooth elliptic curve $R_i\cong E/\langle\epsilon_i\rangle$. These are the only multiple fibers. Since $C'\rightarrow B'$ is étale, the associated fibration $g\colon Y\rightarrow B'$ is smooth. Let $\vartheta_E$ be a nonzero translation-invariant $1$-form on $E$. The $G$-equivariant decomposition 
\[
\Omega_{E\times C}^1\cong \mathscr{O}_{E\times C}\vartheta_E\oplus\pr_C^*\omega_C
\]
descends to a decomposition $\Omega_X^1\cong \mathscr{O}_X\vartheta\oplus\mathscr A$, where $\vartheta$ is nowhere vanishing. Taking determinants identifies $\mathscr A$ with $\omega_X$. Note that $\mathscr{A}$ is actually the saturation of the image of the cotangent map $df\colon f^*\Omega_B^1\rightarrow \Omega_X^1$, which then implies that 
\[
\mathscr{A}\cong f^*\Omega_B^1\biggl(\sum_{i=1}^5 R_i\biggr),
\]
and the required result then follows as $f^*\Omega_B^1\cong \mathscr{O}_X(-2F)$ and $2R_i\sim F$ for any $1\leq i\leq 5$.
\end{proof}

Let $F_0\coloneqq f^{-1}(b_0)$, which is a smooth reduced fiber of $f$.

\begin{lemma}\label{lem:canonical-double-cover}
The morphism $\pi\colon Y\rightarrow X$ is the cyclic double cover associated with $\omega_X$ and a section $a\in H^0(X,\omega_X^{\otimes 2})$ whose zero divisor is $F_0$. Equivalently,
\[
\pi_*\mathscr{O}_Y\cong \mathscr{O}_X\oplus\omega_X^{-1}.
\]
\end{lemma}

\begin{proof}
Let $\rho\colon E\times C\rightarrow X$ be the quotient map. After pullback
along $\rho$, the double cover $\pi\colon Y\rightarrow X$ becomes
$E\times C'\rightarrow E\times C$, which is defined by the line bundle
$\pr_C^*\omega_C$ and the section $\pr_C^*a_C$. On the other hand, the
translation-invariant trivialization of $\omega_E$ gives a $G$-equivariant
identification $\rho^*\omega_X\cong\pr_C^*\omega_C$. Equivariant descent
therefore shows that $\pi$ is defined by $\omega_X$.

The zero divisor of $\pr_C^*a_C$ is $E\times q^{-1}(b_0)$, whose image in
$X$ is the smooth fiber $F_0$. Hence the descended section
$a\in H^0(X,\omega_X^{\otimes2})$ has divisor $F_0$, and
$\pi_*\mathscr{O}_Y\cong \mathscr{O}_X\oplus\omega_X^{-1}$.
\end{proof}

\subsubsection{Construction of the polarization $H$}

We recall the period and index of a genus one curve. Let $T$ be a smooth projective geometrically integral curve of genus one over a field $\mathbb{K}$, and let $J\coloneqq \Pic^0_{T/\mathbb{K}}$ be its Jacobian. The curve $T$ is a torsor under $J$, and its \emph{period} $\operatorname{per}(T)$ is the order of its class $[T]\in H^1(\mathbb{K},J)$. Equivalently, $\operatorname{per}(T)$ is the least positive integer $d$ for which $\Pic^d_{T/\mathbb{K}}(\mathbb{K})$ is nonempty. The \emph{index} $\operatorname{ind}(T)$ is the greatest common divisor of the degrees of the closed points of $T$, or equivalently the positive generator of the image of $\deg\colon\mathrm{CH}_0(T)\rightarrow\mathbb Z$. These equivalent descriptions are recalled in \cite[Definitions~1.1--1.2 and Remark~2.2]{CiperianiKrashen2012}. One always has $\operatorname{per}(T)\mid\operatorname{ind}(T)$ by \cite[p.~669, Proposition~5]{LangTate1958}.

The relation between these invariants is controlled by the Brauer group. The Leray spectral sequence for $\mathbb G_m$ gives an exact sequence
\[0\rightarrow\Pic(T)\rightarrow\Pic_{T/\mathbb{K}}(\mathbb{K})\rightarrow\operatorname{Br}(\mathbb{K})\rightarrow\operatorname{Br}(T),\]
see \cite[Chapter~8, \S1, Proposition~4]{BoschLutkebohmertRaynaud1990} or \cite[Theorem~2.1]{CiperianiKrashen2012}.

If $\mathbb{K}=\C(B)$, then $\operatorname{Br}(\mathbb{K})=0$ by Tsen's theorem \cite[Proposition~6.2.3 and Theorem~6.2.8]{GilleSzamuely2017}. Consequently, we have that $\operatorname{per}(T)=\operatorname{ind}(T)$, so that $T$ has a closed point of degree $\operatorname{per}(T)$ \cite[p.~670]{LangTate1958}.
The closure of the closed point in a proper model over $B$ is an integral multisection of degree $\operatorname{per}(T)$.

\begin{lemma}\label{lem:bisection-and-polarization}
Let $\mathbb{K}=\C(B)$, and let $X_{\mathbb{K}}$ be the generic fiber of $f\colon X\rightarrow B$. Then $X_{\mathbb{K}}$ has period and index two. Consequently, $f$ admits a bisection $D$. For every sufficiently large integer $m$, the divisor $H\coloneqq D+mF$ is ample and satisfies $H\cdot F=2$ and $K_X\cdot H=1$.
\end{lemma}
\begin{proof}
Let $P\coloneqq C\times_B\Spec \mathbb{K}$. Then $P$ is a $G$-torsor over $\mathbb{K}$, and the generic fiber $X_{\mathbb{K}}$ is the contracted product $P\times^G E$, where $G=E[2]$ acts on $E$ by translations. Thus the class of $X_{\mathbb{K}}$ in $H^1(\mathbb{K},E)$ lies in the image of $H^1(\mathbb{K},E[2])$ and is therefore killed by two.

This class is nontrivial. Indeed, if $X_{\mathbb{K}}$ had a $\mathbb{K}$-rational point, then properness would extend it to a section $S$ of $f$. Such a section satisfies $S\cdot F=1$. On the other hand, each multiple fiber has the form $2R_i\sim F$, which would give $1=S\cdot F=2S\cdot R_i$, a contradiction. Hence $X_{\mathbb{K}}$ has period two. By the period--index discussion above, it also has index two, and therefore possesses a closed point of degree two. Its closure in $X$ is an integral bisection $D$ of $f$.

Every reduced fiber of $f$ is irreducible, and $D$ has positive degree on each of them. Thus $D$ is relatively ample over $B$. It follows that $H\coloneqq D+mF$ is ample for every sufficiently large integer $m$. Since $F^2=0$ and $D\cdot F=2$, we have $H\cdot F=2$. Finally, Lemma~\ref{lem:geometry-of-quotient-surface} gives $2K_X\sim F$, and hence $K_X\cdot H=1$.
\end{proof}

\subsection{Proof of Theorem~\ref{thm:polystable-counterexample}}

Now we apply Proposition \ref{prop:parity-obstruction} to the polarized surface $(X,H)$ constructed above. First, we recall the following commutative diagram: 
\[
\begin{tikzcd}[column sep=large,row sep=large]
Y=(E\times C')/G \arrow[r,"\pi"] \arrow[d,"g"']
  & X=(E\times C)/G \arrow[d,"f"] \\
B' \arrow[r,"\beta"']
  & B .
\end{tikzcd}
\]
Now we prove that the polarization $H$ satisfies the hypothesis of Proposition \ref{prop:parity-obstruction}.

\begin{lemma}\label{lem:parity-on-double-cover}
Let $F_Y$ denote the numerical class of a fiber of $g\colon Y\rightarrow B'$, and put $P\coloneqq \pi^*H$. Then $\operatorname{Num}(Y)=\mathbb Z[F_Y]\oplus\mathbb Z[P]$. Consequently, $c_1(\mathscr{L})\cdot\pi^*H$ is even for every line bundle $\mathscr{L}$ on $Y$.
\end{lemma}
\begin{proof}
Set $\mathbb{K}'\coloneqq \C(B')$, and let $T\coloneqq Y_{\mathbb{K}'}$ be the generic fiber of $g$. Since $T$ is obtained from the generic $G$-torsor of $C'\rightarrow B'$ by extension of structure group along $G=E[2]\hookrightarrow E$, its class in $H^1(\mathbb{K}',E)$ is killed by two. This class is nontrivial. Indeed, a $\mathbb{K}'$-point of $T$ would extend by properness to a section of $g$. After pullback to $C'$, such a section would give a $G$-equivariant morphism $C'\rightarrow E$. The assumption $\operatorname{Hom}(J(C'),E)=0$ forces every morphism $C'\rightarrow E$ to be constant, which is incompatible with the translation action of $G$. Thus $T$ has period two and, by the period--index discussion above, index two.

Consider the homomorphism $\delta\colon\operatorname{Num}(Y)\rightarrow\mathbb Z$, given by $\delta([\mathscr L])=c_1(\mathscr L)\cdot F_Y$. Since $T$ has index two, $\operatorname{im}(\delta)\subseteq2\mathbb Z$. On the other hand, $\pi$ identifies a general fiber of $g$ with a general fiber of $f$. Hence Lemma~\ref{lem:bisection-and-polarization} gives $\delta(P)=P\cdot F_Y=H\cdot F=2$, and therefore $\operatorname{im}(\delta)=2\mathbb Z$.

\bigskip
\textbf{Claim.} \textit{$\ker(\delta)=\mathbb Z[F_Y]$.}

\bigskip

Assuming this claim for the time being, let us see how to conclude. The homomorphism $\delta$ has kernel $\mathbb Z[F_Y]$ and image $2\mathbb Z$, while $P$ maps to $2$. It follows that $\operatorname{Num}(Y)=\mathbb Z[F_Y]\oplus\mathbb Z[P]$. Finally, $F_Y\cdot P=2$ and $P^2=(\pi^*H)^2=2H^2$. Hence, writing $c_1(\mathscr{L})=aF_Y+bP$, we obtain $c_1(\mathscr{L})\cdot\pi^*H=2a+2bH^2$, which is even.
\bigskip

\noindent\textit{Proof of the claim.} Let $\rho\colon E\times C'\rightarrow Y$ be the quotient map, and let $\mathscr L$ be a line bundle on $Y$ such that $c_1(\mathscr{L})\cdot F_Y=0$. The restrictions of $\rho^*\mathscr L$ to the curves $E\times\{c\}$ define a morphism $C'\rightarrow\operatorname{Pic}^0(E)\cong E$, which is constant by $\operatorname{Hom}(J(C'),E)=0$; denote its value by $\lambda$. Since $[2]\circ\pr_E\colon E\times C'\rightarrow E$ is $G$-invariant, it descends to a morphism $u\colon Y\rightarrow E$:
\[
\begin{tikzcd}[column sep=large,row sep=large]
E\times C' \arrow[r,"\rho"] \arrow[d,"\pr_E"']
    & Y \arrow[d,"u"] \\
E \arrow[r,"{[2]}"']
    & E
\end{tikzcd}
\]
The surjectivity of $[2]^*$ on $\operatorname{Pic}^0(E)$ allows us to choose $\mathscr M\in\operatorname{Pic}^0(E)$ such that $\lambda\otimes[2]^*\mathscr M$ is trivial. It follows that $\rho^*(\mathscr L\otimes u^*\mathscr M)$ restricts trivially to every $E\times\{c\}$, and hence $\mathscr L\otimes u^*\mathscr M$ is trivial on the generic fiber of $g$. Since $u^*\mathscr M$ is numerically trivial, $\mathscr L$ is numerically equivalent to a vertical divisor. Every fiber of $g$ is smooth and connected, so every vertical divisor is the pullback of a divisor on $B'$. Thus $[\mathscr L]\in\mathbb Z[F_Y]$, proving the claim.
\end{proof}

\begin{proof}[Proof of Theorem~\ref{thm:polystable-counterexample}]
We apply Proposition~\ref{prop:parity-obstruction} to the polarized surface $(X,H)$ constructed above. By Lemma~\ref{lem:geometry-of-quotient-surface}, we have $\Omega_X^1\cong \mathscr{O}_X\vartheta\oplus\omega_X$. Moreover, Lemma~\ref{lem:canonical-double-cover} identifies $\pi\colon Y\rightarrow X$ with the cyclic double cover defined by a section $a\in H^0(X,\omega_X^{\otimes 2})$.

The equation $\eta^2=a$ realizes $Y$ as a degree two spectral cover inside $\Tot(\omega_X)\subset T^*X$. Let $s\in S_X^3$ be the spectral datum whose spectral variety $X_s$ is the disjoint union $\Gamma_\vartheta\sqcup Y$. The two components are disjoint because $\vartheta$ spans the trivial summand of $\Omega_X^1$, whereas $Y$ lies in the summand $\omega_X$. Since $\pi$ is generically \'etale, the datum $s$ is generically regular.

Choose the polarization $H$ given by Lemma~\ref{lem:bisection-and-polarization}. Then $K_X\cdot H=1$, while Lemma~\ref{lem:parity-on-double-cover} shows that $c_1(\mathscr{L})\cdot\pi^*H$ is even for every line bundle $\mathscr{L}$ on $Y$. Thus all the hypotheses of Proposition~\ref{prop:parity-obstruction} are satisfied with $\mathscr A=\omega_X$. It follows that $\mathcal H_{X,H}^{3,\mathrm{ss}}(s)$ is empty. In particular, $\mathcal H_{X,H}^{3,\mathrm{ps}}(s)$ is empty, which proves the theorem.
\end{proof}

\appendix
\section{Equivariant Noether--Lefschetz Theorem}\label{sec: dim 2}\;

The classical Noether--Lefschetz theorem says that for a general surface $S\subset \mathbb{P}^3$ of degree $d\geq 4$, the restriction $\operatorname{Pic}(\mathbb{P}^3)\rightarrow \operatorname{Pic}(S)$ is an isomorphism. Following the arguments from \cite{carlson1983infinitesimal}, we prove an equivariant version for a sufficiently positive line bundle $\mathscr{L}$ on a smooth projective $3$-fold, which is used in the proof of Theorem \ref{thm:sl2-counterexample}.

\subsection{Setup and main results}

Let $\Gamma$ be a finite group acting on a smooth projective variety $Z$ of dimension $2m+1$ with $\pi\colon Z\rightarrow Z/\Gamma$ the quotient map. Let $\mathscr{L}$ be a sufficiently high power of a
$\Gamma$-linearized ample line bundle on $Z$. 
Let $s\in H^0(Z,\mathscr{L})^\Gamma$ define a smooth divisor $i\colon Y\hookrightarrow Z$. Define the ambient part of the middle cohomology by
\[
H_{\mathrm{amb}}^{2m}(Y,\mathbb Q)
\coloneqq 
\operatorname{Im}\left(
i^*\colon H^{2m}(Z,\mathbb Q)\longrightarrow H^{2m}(Y,\mathbb Q)
\right),
\]
and set $H_{\mathrm{amb}}^{p,q}(Y):=H_{\mathrm{amb}}^{2m}(Y,\mathbb C)\cap H^{p,q}(Y)$. Let $Q$ denote the polarization on $H^{2m}(Y,\mathbb Q)$ and define the
variable part by
\[
H_{\mathrm{var}}^{2m}(Y,\mathbb Q)
\coloneqq 
H_{\mathrm{amb}}^{2m}(Y,\mathbb Q)^\perp.
\]

\begin{theorem}[Equivariant Noether--Lefschetz]
	\label{thm: finite group noether lefschetz}
	With the choice of a sufficiently positive $\mathscr{L}$ as above, for a very general $Y\in |\mathscr{L}|^{\Gamma}$,
	\[
	H_{\mathrm{var}}^{2m}(Y,\mathbb Q)
	\cap H^{m,m}(Y)
	=
	0.
	\]
\end{theorem}

Here, very general means outside a countable union of proper Zariski-closed algebraic subsets. 

\begin{corollary}\label{cor: finite group surface noether lefschetz}
	Assume that $\dim Z=3$. 
	For a very general smooth invariant divisor $Y\in |\mathscr{L}|^\Gamma$, restriction induces an isomorphism of $\mathbb Q[\Gamma]$-modules:
	\[
    \begin{tikzcd}[column sep=large]
        i^*\colon 
	\operatorname{NS}(Z)_{\mathbb Q}  \arrow[r,"{\cong}"]
	& 	\operatorname{NS}(Y)_{\mathbb Q}.
    \end{tikzcd}
	\]
\end{corollary}

\begin{proof}
	Combining the Lefschetz $(1,1)$-theorem and Theorem~\ref{thm: finite group noether lefschetz}, we have
	\[
	\operatorname{NS}(Y)_{\mathbb Q}
	=
	H_{\mathrm{amb}}^2(Y,\mathbb Q)\cap H^{1,1}(Y).
	\]
	
	By weak Lefschetz, the restriction map $i^*:H^2(Z,\mathbb Q)\rightarrow H^2(Y,\mathbb Q)$
	is injective and a morphism of Hodge structures.  Applying the Lefschetz $(1,1)$-theorem on $Z$, we obtain the desired isomorphism.
	All the maps are $\Gamma$-equivariant, so the resulting isomorphism
	is an isomorphism of $\mathbb Q[\Gamma]$-modules.
\end{proof}

\subsection{Proof of Theorem \ref{thm: finite group noether lefschetz}}

To prove Theorem \ref{thm: finite group noether lefschetz}, we start with the following lemma.

\begin{lemma}
\label{lem: finite group multiplication}
	Let $\mathscr{F}$ be a $\Gamma$-linearized coherent sheaf on $Z$. After replacing
	$\mathscr{L}$ by a sufficiently high power, the multiplication maps
	\[
	H^0(Z,\mathscr{L})^\Gamma\otimes
	H^0\bigl(Z,\mathscr{F}\otimes \mathscr{L}^{q+1}\bigr)
	\longrightarrow
	H^0\bigl(Z,\mathscr{F}\otimes \mathscr{L}^{q+2}\bigr)
	\]
	are surjective for every $q\geq 0$.
\end{lemma}
\begin{proof}
	After replacing $\mathscr{L}$ by a sufficiently high power, we may assume that $\mathscr{L}$ is the pullback of a line bundle $\overline{\mathscr{L}}$ on $\overline{Z}\coloneqq Z/\Gamma$.
	We then have that $H^0(Z,\mathscr{L})^{\Gamma}\cong H^0(\overline{Z},\overline{\mathscr{L}})$, and the multiplication map in the statement becomes the usual multiplication map
	\[
	H^0(\overline Z,\overline{\mathscr{L}})\otimes
	H^0\bigl(\overline Z,(\pi_*\mathscr{F})\otimes\overline{\mathscr{L}}^{q+1}\bigr)
	\longrightarrow
	H^0\bigl(\overline Z,(\pi_*\mathscr{F})\otimes\overline{\mathscr{L}}^{q+2}\bigr).
	\]
	After choosing $\overline{\mathscr{L}}$ sufficiently positive, these maps are
	surjective for every $q\geq0$ by Mumford's theorem \cite[Theorem 1.8.5]{lazarsfeld2017positivity}. This proves the lemma.
\end{proof}

Let $B\subset |\mathscr{L}|^\Gamma$ be a sufficiently small analytic
neighborhood of $[s]$ contained in the smooth locus, and let $f\colon \mathcal Y_B\longrightarrow B$
be the universal family. Let $\kappa\colon T_{[s]}B\longrightarrow H^1(Y,T_Y)$
be its Kodaira--Spencer map, and set
\[T_Y^\Gamma\coloneqq \operatorname{Im}(\kappa)
\]
to be the deformation directions coming
from the invariant linear system.
The infinitesimal variation of Hodge structure induces $\Gamma$-equivariant maps
\[
T_Y^\Gamma\otimes H^{p,q}(Y)
\longrightarrow H^{p-1,q+1}(Y),
\qquad
\xi\otimes\gamma\longmapsto\nabla_\xi\gamma.
\]
Define the infinitesimally fixed part to be
\[
H_{\mathrm{if},\Gamma}^{m,m}(Y)
:=
\ker\left(
H^{m,m}(Y)\longrightarrow
\operatorname{Hom}\bigl(
T_Y^\Gamma,H^{m-1,m+1}(Y)
\bigr)
\right).
\]

\begin{theorem}
	\label{thm: finite group infinitesimal noether}
	With the choice of a sufficiently positive $\mathscr{L}$ as above, we have
	\[
	H_{\mathrm{if},\Gamma}^{m,m}(Y)
	=
	H_{\mathrm{amb}}^{m,m}(Y).
	\]
\end{theorem}
\begin{proof}
	We follow the proof of the infinitesimal M.~Noether theorem in
	\cite[\S\,3(a)]{carlson1983infinitesimal}, replacing the full linear system
	by the invariant linear system. The inclusion $H_{\mathrm{amb}}^{m,m}(Y)
	\subset
	H_{\mathrm{if},\Gamma}^{m,m}(Y)$ is clear, since every class in $H_{\mathrm{amb}}^{m,m}(Y)$ is the restriction of a class of type $(m,m)$ on $Z$, thus it remains of that type along any deformation inside $Z$.
	
	We prove the opposite inclusion. Set
	\[
	M_q
	\coloneqq 
	H^0\bigl(Z,\omega_Z\otimes \mathscr{L}^{q+1}\bigr)
	=
	H^0\bigl(Z,\omega_Z((q+1)Y)\bigr).
	\]
	Applying Lemma~\ref{lem: finite group multiplication} to \(\mathscr{F}=\omega_Z\), we
	obtain surjective multiplication maps
	\begin{equation}\label{eq: finite group multiplication canonical}
		H^0(Z,\mathscr{L})^\Gamma\otimes M_q
		\twoheadrightarrow
		M_{q+1}
	\end{equation}
	for every \(q\geq0\).	By \cite[(3.a.8)]{carlson1983infinitesimal}, we have the following short exact sequence
	\begin{equation}\label{eq: finite group residue}
		0
		\longrightarrow
		H_{\mathrm{prim}}^{2m+1-q,q}(Z)
		\longrightarrow
		\frac{
			H^0\bigl(Z,\omega_Z((q+1)Y)\bigr)
		}{
			dH^0\bigl(Z,\Omega_Z^{2m}(qY)\bigr)
			+
			H^0\bigl(Z,\omega_Z(qY)\bigr)
		}
		\xrightarrow{\operatorname{Res}}
		H_{\mathrm{var}}^{2m-q,q}(Y)
		\longrightarrow
		0.
	\end{equation}
	All the maps in \eqref{eq: finite group residue} are $\Gamma$-equivariant.
	
	Let $t\in H^0(Z,\mathscr{L})^\Gamma$, and let $\xi_t\in T_Y^\Gamma$
	be the corresponding Kodaira--Spencer class. An element
	$u\in M_q$ determines the meromorphic top form $\frac{u}{s^{q+1}}$ on $Z$, whose residue defines a class in $H_{\mathrm{var}}^{2m-q,q}(Y)$.	Consider the first-order deformation of \(Y\) defined by $s+\epsilon t$.
	We have
	\[
	\left.
	\frac{d}{d\epsilon}
	\right|_{\epsilon=0}
	\frac{u}{(s+\epsilon t)^{q+1}}
	=
	-(q+1)\frac{tu}{s^{q+2}}.
	\]
	Consequently, the infinitesimal variation of Hodge structure satisfies
	\begin{equation}\label{eq: finite group infinitesimal residue}
		\nabla_{\xi_t}
		\operatorname{Res}\left(\frac{u}{s^{q+1}}\right)
		=
		-(q+1)
		\operatorname{Res}\left(\frac{tu}{s^{q+2}}\right).
	\end{equation}
	Combining \eqref{eq: finite group multiplication canonical}, \eqref{eq: finite group residue}, and  \eqref{eq: finite group infinitesimal residue}, we have that the morphism 
	\begin{equation}\label{eq: finite group ivhs surjective}
		T_Y^\Gamma\otimes
		H_{\mathrm{var}}^{2m-q,q}(Y)
		\longrightarrow
		H_{\mathrm{var}}^{2m-q-1,q+1}(Y)
	\end{equation}
	is surjective.	Take any class $\gamma\in H_{\mathrm{if},\Gamma}^{m,m}(Y)\cap H_{\mathrm{var}}^{m,m}(Y)$.
	By \eqref{eq: finite group ivhs surjective}, every class $\eta\in H_{\mathrm{var}}^{m,m}(Y)$
	can be written as
	\[
	\eta=\sum_i\nabla_{\xi_i}\beta_i,
	\qquad
	\xi_i\in T_Y^\Gamma,
	\quad
	\beta_i\in H_{\mathrm{var}}^{m+1,m-1}(Y).
	\]
	Using adjointness, see
	\cite[p.~178, the displayed formula following (3.a.19)]
	{carlson1983infinitesimal}, we have that 
	\[
	\begin{aligned}
		Q(\gamma,\eta)
		=
		\sum_i
		Q(\gamma,\nabla_{\xi_i}\beta_i)
		=
		-\sum_i
		Q(\nabla_{\xi_i}\gamma,\beta_i)=
		0.
	\end{aligned}
	\]
	Therefore, $\gamma=0$, and the theorem is proved.
\end{proof}

\begin{proof}[Proof of Theorem \ref{thm: finite group noether lefschetz}]
	By the algebraicity theorem of Cattani--Deligne--Kaplan
	\cite[Theorem~1.1 and Corollary~1.3]
	{CattaniDeligneKaplan1995}, the Hodge locus of each fixed integral
	class is algebraic, so that the locus of
	$b\in B$ such that  $H_{\mathrm{var}}^{2m}(Y_b,\mathbb Q)
	\cap H^{m,m}(Y_b)\neq0$
	is a countable union of closed algebraic subsets of $B$.
	
	It suffices to show that the closed subsets are not all of $B$.
	Suppose otherwise, then, after possibly shrinking $B$ to make it simply connected, there exists a nonzero $\gamma_b\in H_{\mathrm{var}}^{2m}(Y_b,\mathbb Q)$ whose parallel transport remains of type $(m,m)$.
	Therefore, $\gamma_b\in H_{\mathrm{if},\Gamma}^{m,m}(Y_b)$.
	By Theorem~\ref{thm: finite group infinitesimal noether},
	we have that $\gamma_b=0$, contradiction.
\end{proof}

\bibliographystyle{plain}
\bibliography{references}

@article{CattaniDeligneKaplan1995,
  author  = {Cattani, Eduardo and Deligne, Pierre and Kaplan, Aroldo},
  title   = {On the locus of {H}odge classes},
  journal = {Journal of the American Mathematical Society},
  volume  = {8},
  number  = {2},
  pages   = {483--506},
  year    = {1995},
  doi     = {10.1090/S0894-0347-1995-1273413-3},
  eprint  = {alg-geom/9402009},
  archivePrefix = {arXiv}
}

@article{LangTate1958,
  author  = {Lang, Serge and Tate, John},
  title   = {Principal Homogeneous Spaces over Abelian Varieties},
  journal = {American Journal of Mathematics},
  volume  = {80},
  number  = {3},
  year    = {1958},
  pages   = {659--684},
  doi     = {10.2307/2372778}
}

@article{CiperianiKrashen2012,
  author        = {Ciperiani, Mirela and Krashen, Daniel},
  title         = {Relative {Brauer} Groups of Genus 1 Curves},
  journal       = {Israel Journal of Mathematics},
  volume        = {192},
  number        = {2},
  year          = {2012},
  pages         = {921--949},
  doi           = {10.1007/s11856-012-0057-5},
  eprint        = {math/0701614},
  archivePrefix = {arXiv}
}

@book{BoschLutkebohmertRaynaud1990,
  author    = {Bosch, Siegfried and L{\"u}tkebohmert, Werner
               and Raynaud, Michel},
  title     = {N{\'e}ron Models},
  series    = {Ergebnisse der Mathematik und ihrer Grenzgebiete.
               3. Folge},
  volume    = {21},
  publisher = {Springer-Verlag},
  address   = {Berlin},
  year      = {1990},
  doi       = {10.1007/978-3-642-51438-8}
}

@article{maulik2021endoscopic,
  author  = {Maulik, Davesh and Shen, Junliang},
  title   = {Endoscopic decompositions and the {Hausel--Thaddeus} conjecture},
  journal = {Forum of Mathematics, Pi},
  volume  = {9},
  pages   = {Paper No. e8, 49},
  year    = {2021},
  doi     = {10.1017/fmp.2021.7}
}

@article{Namikawa1985,
  author  = {Namikawa, Yukihiko},
  title   = {Periods of {Enriques} surfaces},
  journal = {Mathematische Annalen},
  volume  = {270},
  number  = {2},
  pages   = {201--222},
  year    = {1985},
  doi     = {10.1007/BF01456182}
}

@book{Hartshorne1977,
  author    = {Hartshorne, Robin},
  title     = {Algebraic Geometry},
  series    = {Graduate Texts in Mathematics},
  volume    = {52},
  publisher = {Springer-Verlag},
  address   = {New York--Heidelberg},
  year      = {1977},
  doi       = {10.1007/978-1-4757-3849-0}
}

@book{GilleSzamuely2017,
  author    = {Gille, Philippe and Szamuely, Tam{\'a}s},
  title     = {Central Simple Algebras and {Galois} Cohomology},
  edition   = {Second},
  series    = {Cambridge Studies in Advanced Mathematics},
  volume    = {165},
  publisher = {Cambridge University Press},
  address   = {Cambridge},
  year      = {2017},
  doi       = {10.1017/9781316661277}
}

@Article{Flenner1981,
  author     = {Flenner, Hubert},
  title      = {Divisorenklassengruppen quasihomogener {S}ingularit\"aten},
  journal    = {J. Reine Angew. Math.},
  year       = {1981},
  volume     = {328},
  pages      = {128--160},
  fjournal   = {Journal f\"ur die Reine und Angewandte Mathematik},
  issn       = {0075-4102},
  mrclass    = {13F15 (13G05 14B05 14C99)},
  mrnumber   = {636200},
  mrreviewer = {Gerd Faltings},
  url        = {https://doi.org/10.1515/crll.1981.328.128},
}

@article{carlson1983infinitesimal,
  title={Infinitesimal variations of Hodge structure (I)},
  author={Carlson, James and Green, Mark and Griffiths, Phillip and Harris, Joe},
  journal={Compositio Mathematica},
  volume={50},
  number={2-3},
  pages={109--205},
  year={1983}
}

@Article{RavindraSrinivas2009,
  author  = {Ravindra, Girivaru V. and Srinivas, Vasudevan},
  journal = {J. Algebra},
  title   = {The {N}oether--{L}efschetz theorem for the divisor class group},
  year    = {2009},
  number  = {9},
  pages   = {3373--3391},
  volume  = {322},
  doi     = {10.1016/j.jalgebra.2008.09.003},
}

@book{lazarsfeld2017positivity,
  title={Positivity in algebraic geometry I: Classical setting: line bundles and linear series},
  author={Lazarsfeld, Robert K},
  volume={48},
  year={2017},
  publisher={Springer}
}

@misc {SheshmaniWangXia2026,
    AUTHOR = {Sheshmani, Artan and Wang, Jianping and Xia, Xiaopeng},
     TITLE = {On the image of {H}itchin morphism for some classical groups
              on algebraic surfaces},
      NOTE = {arXiv:2606.17505},
      YEAR = {2026},
       DOI = {10.48550/arXiv.2606.17505},
}

@Article{Hartshorne1980,
  author     = {Hartshorne, Robin},
  title      = {Stable reflexive sheaves},
  journal    = {Math. Ann.},
  year       = {1980},
  volume     = {254},
  number     = {2},
  pages      = {121--176},
  coden      = {MAANA3},
  doi        = {10.1007/BF01467074},
  fjournal   = {Mathematische Annalen},
  issn       = {0025-5831},
  mrclass    = {14F05 (14D22)},
  mrnumber   = {597077},
  mrreviewer = {Klaus Hulek},
  url        = {http://dx.doi.org/10.1007/BF01467074},
}

@article{BeauvilleNarasimhanRamanan1989,
  author   = {Beauville, Arnaud and Narasimhan, Mudumbai S. and Ramanan, Sundararaman},
  title    = {Spectral curves and the generalised {T}heta divisor},
  journal  = {J. Reine Angew. Math.},
  fjournal = {Journal f{\"u}r die reine und angewandte Mathematik},
  volume   = {398},
  year     = {1989},
  pages    = {169--179},
  doi      = {10.1515/crll.1989.398.169}
}

@article{ChenNgo2020,
  author   = {Chen, Tsao-Hsien and Ng{\^o}, Bao Ch{\^a}u},
  title    = {On the {H}itchin morphism for higher-dimensional varieties},
  journal  = {Duke Math. J.},
  fjournal = {Duke Mathematical Journal},
  volume   = {169},
  year     = {2020},
  number   = {10},
  pages    = {1971--2004},
  doi      = {10.1215/00127094-2019-0085}
}

@article{HeLiu2024,
  author   = {He, Siqi and Liu, Jie},
  title    = {On the spectral variety for rank two {H}iggs bundles},
  journal  = {Proc. Lond. Math. Soc. (3)},
  fjournal = {Proceedings of the London Mathematical Society. Third Series},
  volume   = {129},
  year     = {2024},
  number   = {5},
  pages    = {Paper No. e70004, 45},
  doi      = {10.1112/plms.70004}
}

@misc{HeLiuMok2024,
  author        = {He, Siqi and Liu, Jie and Mok, Ngaiming},
  title         = {The spectral base and quotients of bounded symmetric domains by cocompact lattices},
  year          = {2024},
  eprint        = {2401.15852},
  archivePrefix = {arXiv},
  primaryClass  = {math.AG},
  note          = {arXiv:2401.15852}
}

@article{Hitchin1987,
  author   = {Hitchin, Nigel J.},
  title    = {Stable bundles and integrable systems},
  journal  = {Duke Math. J.},
  fjournal = {Duke Mathematical Journal},
  volume   = {54},
  year     = {1987},
  number   = {1},
  pages    = {91--114},
  doi      = {10.1215/S0012-7094-87-05408-1}
}

@article{Huynh2025,
  author        = {Huynh, Matthew},
  title         = {The {H}itchin morphism for certain surfaces fibered over a curve},
  journal       = {Pure Appl. Math. Q.},
  fjournal      = {Pure and Applied Mathematics Quarterly},
  volume        = {21},
  year          = {2025},
  number        = {6},
  pages         = {2257--2286},
  eprint        = {2504.01874},
  archivePrefix = {arXiv},
  primaryClass  = {math.AG}
}

@misc{PatelWeissmann2026a,
  author        = {Patel, Aryaman and Weissmann, Dario},
  title         = {Stratifying moduli spaces of {H}iggs bundles and the {H}itchin morphism},
  year          = {2026},
  eprint        = {2601.08597},
  archivePrefix = {arXiv},
  primaryClass  = {math.AG},
  note          = {arXiv:2601.08597}
}

@misc{PatelWeissmann2026b,
  author        = {Patel, Aryaman and Weissmann, Dario},
  title         = {The {H}itchin morphism for {$K$}-trivial varieties},
  year          = {2026},
  eprint        = {2604.03217},
  archivePrefix = {arXiv},
  primaryClass  = {math.AG},
  note          = {arXiv:2604.03217}
}

@Article{Kawamata1981,
  author     = {Kawamata, Yujiro},
  journal    = {Compositio Math.},
  title      = {Characterization of abelian varieties},
  year       = {1981},
  issn       = {0010-437X,1570-5846},
  number     = {2},
  pages      = {253--276},
  volume     = {43},
  fjournal   = {Compositio Mathematica},
  mrclass    = {14J10 (32J15)},
  mrnumber   = {622451},
  mrreviewer = {Daniel\ Comenetz},
  url        = {http://www.numdam.org/item?id=CM_1981__43_2_253_0},
}

@Article{Sun2026,
  author  = {Hao Sun},
  journal = {arXiv preprint arXiv:2609.06933},
  title   = {Cyclic Spectral Data},
  year    = {2026},
}

@misc{SheshmaniXiaYuan2025,
  author        = {Sheshmani, Artan and Xia, Xiaopeng and Yuan, Beihui},
  title         = {On the normality of commuting scheme for general linear {L}ie algebra},
  year          = {2025},
  eprint        = {2505.13013},
  archivePrefix = {arXiv},
  primaryClass  = {math.AG},
  note          = {arXiv:2505.13013}
}

@article{Simpson1992,
  author   = {Simpson, Carlos T.},
  title    = {Higgs bundles and local systems},
  journal  = {Inst. Hautes {\'E}tudes Sci. Publ. Math.},
  fjournal = {Publications Math{\'e}matiques de l'Institut des Hautes {\'E}tudes Scientifiques},
  volume   = {75},
  year     = {1992},
  pages    = {5--95},
  doi      = {10.1007/BF02699491}
}

@article{SongSun2024,
  author   = {Song, Lei and Sun, Hao},
  title    = {On the image of {H}itchin morphism for algebraic surfaces: the case {$\mathrm{GL}_n$}},
  journal  = {Int. Math. Res. Not. IMRN},
  fjournal = {International Mathematics Research Notices. IMRN},
  year     = {2024},
  number   = {1},
  pages    = {492--514},
  doi      = {10.1093/imrn/rnad043}
}

@article{SongXiaXu2023,
  author   = {Song, Lei and Xia, Xiaopeng and Xu, Jinxing},
  title    = {A higher-dimensional {C}hevalley restriction theorem for orthogonal groups},
  journal  = {Adv. Math.},
  fjournal = {Advances in Mathematics},
  volume   = {426},
  year     = {2023},
  pages    = {Paper No. 109104},
  doi      = {10.1016/j.aim.2023.109104}
}

@misc{StacksProject,
  author = {{The Stacks Project Authors}},
  title  = {The Stacks Project},
  url    = {https://stacks.math.columbia.edu}
}

@article{Donaldson1987,
  author   = {Donaldson, Simon Kirwan},
  title    = {Twisted harmonic maps and the self-duality equations},
  journal  = {Proc. London Math. Soc. (3)},
  fjournal = {Proceedings of the London Mathematical Society. Third Series},
  volume   = {55},
  year     = {1987},
  number   = {1},
  pages    = {127--131},
  doi      = {10.1112/plms/s3-55.1.127}
}

@article{Hitchin1987SelfDuality,
  author   = {Hitchin, Nigel J.},
  title    = {The self-duality equations on a {R}iemann surface},
  journal  = {Proc. London Math. Soc. (3)},
  fjournal = {Proceedings of the London Mathematical Society. Third Series},
  volume   = {55},
  year     = {1987},
  number   = {1},
  pages    = {59--126},
  doi      = {10.1112/plms/s3-55.1.59}
}

@article{Corlette1988,
  author   = {Corlette, Kevin},
  title    = {Flat {$G$}-bundles with canonical metrics},
  journal  = {J. Differential Geom.},
  fjournal = {Journal of Differential Geometry},
  volume   = {28},
  year     = {1988},
  number   = {3},
  pages    = {361--382}
}

@article{Simpson1988,
  author   = {Simpson, Carlos T.},
  title    = {Constructing variations of {H}odge structure using {Yang--Mills} theory and applications to uniformization},
  journal  = {J. Amer. Math. Soc.},
  fjournal = {Journal of the American Mathematical Society},
  volume   = {1},
  year     = {1988},
  number   = {4},
  pages    = {867--918},
  doi      = {10.2307/1990994}
}

@incollection{Simpson1991,
  author    = {Simpson, Carlos T.},
  title     = {The ubiquity of variations of {H}odge structure},
  booktitle = {Complex geometry and Lie theory},
  series    = {Proc. Sympos. Pure Math.},
  volume    = {53},
  publisher = {Amer. Math. Soc.},
  address   = {Providence, RI},
  year      = {1991},
  pages     = {329--348}
}

@article{Simpson1994a,
  author   = {Simpson, Carlos T.},
  title    = {Moduli of representations of the fundamental group of a smooth projective variety. {I}},
  journal  = {Inst. Hautes {\'E}tudes Sci. Publ. Math.},
  fjournal = {Publications Math{\'e}matiques de l'Institut des Hautes {\'E}tudes Scientifiques},
  volume   = {79},
  year     = {1994},
  pages    = {47--129}
}

@incollection{Simpson1997,
  author    = {Simpson, Carlos T.},
  title     = {The {Hodge} filtration on nonabelian cohomology},
  booktitle = {Algebraic geometry---Santa Cruz 1995},
  series    = {Proc. Sympos. Pure Math.},
  volume    = {62},
  publisher = {Amer. Math. Soc.},
  address   = {Providence, RI},
  year      = {1997},
  pages     = {217--281}
}

@Book{Lazarsfeld2004,
  author    = {Lazarsfeld, Robert},
  publisher = {Berlin: Springer},
  title     = {Positivity in algebraic geometry. {I}. {Classical} setting: line bundles and linear series},
  year      = {2004},
  isbn      = {3-540-22533-1},
  series    = {Ergeb. Math. Grenzgeb., 3. Folge},
  volume    = {48},
  fseries   = {Ergebnisse der Mathematik und ihrer Grenzgebiete. 3. Folge},
  issn      = {0071-1136},
  language  = {English},
  zbl       = {1093.14501},
  zbmath    = {2134816},
}

@Article{Ma2019,
  author     = {Ma, Linquan},
  journal    = {Comm. Algebra},
  title      = {Maximal {C}ohen-{M}acaulay modules over certain {S}egre products},
  year       = {2019},
  issn       = {0092-7872,1532-4125},
  number     = {6},
  pages      = {2488--2493},
  volume     = {47},
  doi        = {10.1080/00927872.2018.1444173},
  fjournal   = {Communications in Algebra},
  mrclass    = {13C14 (13H10)},
  mrnumber   = {3957111},
  mrreviewer = {J.\ K.\ Verma},
  url        = {https://doi.org/10.1080/00927872.2018.1444173},
}

@Book{Serre1965,
  author     = {Serre, Jean-Pierre},
  publisher  = {Springer-Verlag, Berlin-New York},
  title      = {Alg{\`e}bre locale. {M}ultiplicit{\'e}s},
  year       = {1965},
  note       = {Cours au Coll{\`e}ge de France, 1957--1958, r{\'e}dig{\'e} par Pierre Gabriel, Seconde {\'e}dition, 1965},
  series     = {Lecture Notes in Mathematics},
  volume     = {11},
  mrclass    = {13.95 (14.08)},
  mrnumber   = {0201468},
  mrreviewer = {M. Nagata},
  pages      = {vii+188 pp. (not consecutively paged)},
}

@Article{KleimanLandolfi1971,
  author     = {Kleiman, Steven L. and Landolfi, John},
  journal    = {Compositio Math.},
  title      = {Geometry and deformation of special {S}chubert varieties},
  year       = {1971},
  issn       = {0010-437X,1570-5846},
  pages      = {407--434},
  volume     = {23},
  fjournal   = {Compositio Mathematica},
  mrclass    = {14M15 (14M05)},
  mrnumber   = {314855},
  mrreviewer = {Allen\ B.\ Altman},
}

\end{document}